\documentclass[a4paper,12pt,intlimits]{amsart}
\usepackage{amssymb}

\usepackage[english]{babel}

\usepackage{amsmath}

\numberwithin{equation}{section}

\newtheorem{thm}{Theorem}[section]
\newtheorem{rem}[thm]{Remark}

\newtheorem{prop}[thm]{Proposition}
\newtheorem{lem}[thm]{Lemma}

\newcommand{\R}{\mathbb{R}}

\newcommand{\Z}{\mathbb{Z}}

\newcommand{\al}{\alpha}

\newcommand{\xs}{\overline{x}}

\renewcommand{\theta}{\vartheta}
\renewcommand{\epsilon}{\varepsilon}
\newcommand{\ep}{\epsilon}
\newcommand{\I}{\mathcal{I}_s}

\newcommand{\ueu}{u_{\epsilon,1}}
\newcommand{\ued}{u_{\epsilon,2}}

\newcommand{\cs}{\overline{c}}
\newcommand{\vs}{\overline{v}}

\renewcommand{\leq}{\leqslant}
\renewcommand{\le}{\leqslant}
\renewcommand{\geq}{\geqslant}
\renewcommand{\ge}{\geqslant}

\newcommand{\beq}{\begin{equation}}
\newcommand{\eeq}{\end{equation}}
\newcommand{\beqs}{\begin{equation*}}
\newcommand{\eeqs}{\end{equation*}}
\newcommand{\beqa}{\begin{eqnarray}}
\newcommand{\eeqa}{\end{eqnarray}}
\newcommand{\beqas}{\begin{eqnarray*}}
\newcommand{\eeqas}{\end{eqnarray*}}

\title[Crystal dislocations with different orientations 
and collisions]{Crystal dislocations with \\different orientations
and collisions}
\author{Stefania Patrizi and Enrico Valdinoci}
\thanks{The authors have been supported by the
ERC grant 277749 ``EPSILON Elliptic
Pde's and Symmetry of Interfaces and Layers for Odd Nonlinearities''}

\address{Weierstra{\ss} Institut f{\"u}r Angewandte und Stochastik,
Mohrenstrasse 39, D-10117 Berlin, Germany}
\email{Stefania.Patrizi@wias-berlin.de} 
\email{Enrico.Valdinoci@wias-berlin.de}


\begin{document}

\maketitle

\noindent{\footnotesize{\sc Abstract.}
{\sf
We study a parabolic differential equation whose solution
represents the atom dislocation in a crystal for
a general type of Peierls-Nabarro model with possibly long range
interactions and an external stress. Differently
from the previous literature, we treat here the case in which
such dislocation is not the superpositions of transitions
all occurring with the same orientations (i.e. opposite orientations
are allowed as well).

We show that, at a long time scale, and at a macroscopic space scale,
the dislocations have the tendency to concentrate as pure jumps
at points which evolve in time, driven by the external
stress and by a singular potential. Due to differences in the
dislocation orientations, these points may collide in finite time.

More precisely, we consider the evolutionary equation
$$ (v_\ep)_t=\displaystyle\frac{1}{\ep}\left(
\I v_\ep-\displaystyle\frac{1}{\ep^{2s}}W'(v_\ep)+\sigma(t,x)\right),$$
where~$v_\ep=v_\ep(t,x)$ is the atom dislocation fuction at time~$t>0$
at the point~$x\in\R$, $\I$ is an integro-differential
operator of order~$2s\in(0,2)$, $W$ is a periodic potential,
$\sigma$ is an external stress
and~$\ep>0$ is a small parameter that takes into account the small
periodicity scale of the crystal.

We suppose that~$v_\ep(0,x)$ is the superposition
of~$N-K$ transition layers in the positive direction and~$K$
in the negative one (with~$K\in\{0,\dots,N\}$);
more precisely, we fix points~$x_1^0<\dots<x_N^0$ and we
take
$$ v_\ep(0,x)= \displaystyle\frac{\ep^{2s}}{W''(0)}
\sigma(0,x)+\displaystyle\sum_{i=1}^N u\left(\zeta_i
\displaystyle\frac{x-x_i^0}{\ep}\right).$$
Here~$\zeta_i$ is either~$-1$ or~$1$, depending
on the orientation of the transition layer~$u$, which
in turn solves the stationary equation~$\I u=W'(u)$.

We show that our problem possesses a unique solution
and that, as~$\ep\to0^+$, it approaches the sum of Heaviside functions~$H$
with different orientations centered at points~$x_i(t)$, namely
$$ \sum_{i=1}^N H(\zeta_i(x-x_i(t))).$$
The point~$x_i$ evolves in time from~$x_i^0$, being subject to the external
stress and a singular potential, which may be either attractive or repulsive,
according to the different orientation of the transitions: more precisely,
the speed~$\dot x_i$ is proportional to
$$ \sum_{j\neq i}\zeta_i\zeta_j
\frac{x_i-x_j}{2s |x_i-x_j|^{1+2s}}-\zeta_i\sigma(t,x_i).$$
The evolution of such dynamical system may lead to collisions in finite
time. We give a detailed description of such collisions when~$N=2,3$
and we show that the solution itself keeps track of such collisions: indeed,
at the collision time~$T_c$ the two opposite dislocations
have the tendency to annihilate each other and make
the dislocation vanish, but only outside
the collision point~$x_c$, according to the formulas 
\begin{eqnarray*}
&&\displaystyle\lim_{t\rightarrow T_c^-}\lim_{\ep\rightarrow 0^+}v_\ep(t,x)=0
\quad {\mbox{ when $x\ne x_c$,}}\\{\mbox{and }}&&
\displaystyle\limsup_{t\rightarrow T_c^-\atop\ep\rightarrow 0^+} v_\ep(t,x_c)\geq1.\end{eqnarray*}
We also study some specific cases of~$N$ dislocation layers, namely when
two dislocations are initially very close and when the dislocations
are alternate.

To the best of our knowledge, the results obtained are new even 
in the model case~$s=1/2$.
}}
\bigskip

\tableofcontents

\section{Introduction}

The goal of this paper is to study an evolutionary partial-integro-differential equation
and a system of ordinary differential equations that arise in
the Peierls-Nabarro model for atoms dislocation in crystals.

We refer to~\cite{Nab97} for a survey of the
Peierls-Nabarro model. See also Section~2 in~\cite{dpv}
for some basic physical derivation. 

The main goal of the evolutionary equation associated to
the Peierls-Nabarro model is to study the asymptotic behavior
of the solution~$v_\ep$, which represents the atom dislocation function,
in terms of~$\ep$, which in turn
represents the size of the crystal scale.
A suitable parabolic scaling is involved in the equation, and so
the asymptotics as~$\ep\to0^+$ corresponds simultaneously
to the long time and macroscopic space scale behavior.\medskip

Roughly speaking, in this paper we will consider initial configurations
in which the dislocation transitions occurs at some given points.
Differently from the existing literature, the initial dislocations
are not assumed to have all the same orientation. \medskip

We will show that, at a long time and macroscopic scale range,
the solution will behave as the superposition of sharp interfaces.

These interfaces move in time according to an external stress
and an interaction potential. As a main novelty with respect
to the existing literature, we will show that in this case
the potential has two opposite tendences, i.e. it
is repulsive among dislocations with the same orientations
and attractive among dislocations with opposite orientations. \medskip

In configurations in which the attractive feature of the potential
prevails, the dislocation with opposite orientations may collide
one with the other. Therefore we also give some explicit 
results about collisions in concrete cases.\medskip

Let us now formally describe the mathematical
framework that we deal with.
We consider the problem 
\beq\label{vepeq}\begin{cases}
(v_\ep)_t=\displaystyle\frac{1}{\ep}\left(
\I v_\ep-\displaystyle\frac{1}{\ep^{2s}}W'(v_\ep)+\sigma(t,x)\right)&\text{in }(0,+\infty)\times\R\\
v_\ep(0,\cdot)=v_\ep^0&\text{on }\R,
\end{cases}\eeq
where $\ep>0$ is a small scale parameter, $W$ is a periodic potential and  $\I$ is the so-called fractional Laplacian 
 of any order $2s\in(0,2)$.  Precisely, given $\varphi\in C^2(\R^N)\cap L^\infty(\R^N)$, let us define 
\beq\label{slapla} \I[\varphi](x):=PV\displaystyle\int_{\R^N}\displaystyle\frac{\varphi(x+y)-\varphi(x)}{|y|^{N+2s}}dy,\eeq
where $PV$ stands for the principal value of the integral. We refer to \cite{s} and   \cite{dnpv} for a basic introduction to the fractional Laplace operator. 
On the potential $W$ we assume 
\begin{equation}\label{W}
\begin{cases}W\in C^{3,\alpha}(\R)& \text{for some }0<\alpha<1\\
W(v+1)=W(v)& \text{for any } v\in\R\\
W=0& \text{on }\Z\\
W>0 & \text{on }\R\setminus\Z\\
W''(0)>0.\\
\end{cases}
\end{equation}
The function $\sigma$ satisfies: 

\begin{equation}\label{sigmaassump}
\begin{cases}
\sigma \in BUC([0,+\infty)\times\R)\quad\text{and for some }M>0\text{ and }\alpha\in(s,1)\\
\|\sigma_x\|_{L^\infty([0,+\infty)\times\R)}+\|\sigma_t\|_{L^\infty([0,+\infty)\times\R)}\leq M\\
|\sigma_x(t,x+h)-\sigma_x(t,x)|\leq M|h|^\alpha,\quad\text{for every }x,h\in\R \text{ and }t\in[0,+\infty).
\end{cases}
\end{equation}
We assume the initial condition in \eqref{vepeq} to be a superposition of  transition layers. Precisely, let us introduce the so-called basic layer solution $u$ associated to 
$\I$, that  is the solution  of 
\begin{equation}\label{u}
\begin{cases}\I(u)=W'(u)&\text{in}\quad \R\\
u'>0&\text{in}\quad \R\\
\displaystyle\lim_{x\rightarrow-\infty}u(x)=0,\quad\displaystyle\lim_{x\rightarrow+\infty}u(x)=1,\quad u(0)=\displaystyle\frac{1}{2}.
\end{cases}
\end{equation}
The existence of a unique solution 
of \eqref{u} is proven in \cite{cs}. The name layer solution is motivated by the fact that $u$ approaches the limits 0 and 1 at $\pm\infty$. 
Asymptotic estimates  on the decay of $u$ are proven in \cite{psv}, finer estimates are given in \cite{dpv} and  \cite{dfv} respectively when $s\in\left[\frac{1}{2},1\right)$ 
and $s\in\left(0,\frac{1}{2}\right)$. The case $s=\frac{1}{2}$ was already treated in \cite{gonzalezmonneau}. 

Given $x_1^0<x_2^0<...<x_N^0$, 
 we say that the function $u\left(\frac{x-x_i^0}{\ep}\right)$ 
is a transition layer centered at $x_i^0$ and positively oriented. 
Similarly, we say that the function
$u\left(\frac{x_i^0-x}{\ep}\right)-1$
is a transition layer centered at $x_i^0$ and negatively oriented.

Notice that the positively oriented transition layer connects
the ``rest states''~$0$ and~$1$, while the negatively oriented one
connects~$0$ with~$-1$.\medskip

We consider as initial condition in \eqref{vepeq} the state obtained by superposing 
$N$ copies of the transition layer, centered at $x_1^0,..., x_N^0$, $N-K$ of them positively oriented and the remaining $K$ negative oriented, that is 
\beq\label{vep0}
v_\ep^0(x)=\displaystyle\frac{\ep^{2s}}{\beta}
\sigma(0,x)+\displaystyle\sum_{i=1}^N u\left(\zeta_i
\displaystyle\frac{x-x_i^0}{\ep}\right)-K,
\eeq
where $\zeta_1,...,\zeta_N\in\{-1,1\}$, $\displaystyle\sum_{i=1}^N(\zeta_i)^-=K$, $0\leq K\leq N$ and
\beq\label{beta}\beta:=W''(0)>0.\eeq

Let us introduce the solution $(x_i(t))_{i=1,...,N}$ to the system 
\beq\label{dynamicalsysNintro}\begin{cases} \dot{x}_i=\gamma\left(
\displaystyle\sum_{j\neq i}\zeta_i\zeta_j 
\displaystyle\frac{x_i-x_j}{2s |x_i-x_j|^{1+2s}}-\zeta_i\sigma(t,x_i)\right)&\text{in }(0,T_c)\\
 x_i(0)=x_i^0,
\end{cases}\eeq
where \beq\label{gamma}\gamma:=\left(\displaystyle\int_\R (u'(x))^2 dx\right)^{-1},\eeq
with $u$ solution of \eqref{u} and $(0,T_c)$ is the maximal interval where the system \eqref{dynamicalsysNintro} is well defined, i.e. where $x_i\neq x_j$ for any $i\neq j$. 
Therefore, $0<T_c\leq+\infty$ is the first time when a collision between two particles occurs, more precisely $T_c$ is such that: there exist $i_0,j_0$ with $i_0\neq j_0$ such that 
$x_{i_0}(T_c)=x_{j_0}(T_c)$ and $x_i(t)\neq x_j(t)$ for any $t\in[0,T_c)$ and any $i,j$. 

We remark that~\eqref{dynamicalsysNintro} is a gradient system, i.e.
it can be written as
$$\dot x_i (t)=-\partial_i V\big(t,x_1(t),\ldots,x_N(t)\big),$$ with
\begin{eqnarray*}
&&V(t,x_1,\ldots,x_N):=V_0(x_1,\ldots,x_n)
+\displaystyle\sum_{i=1}^N \zeta_i
\Sigma(t,x_i) ,\\
&& V_0(x_1,\ldots,x_n):=\left\{
\begin{matrix}
\displaystyle\frac{\gamma}{2s\,(2s-1)}
\displaystyle\sum_{1\le i\ne j\le N}\zeta_i\zeta_j |x_j-x_i|^{1-2s}
& {\mbox{ if }} s\neq 1/2,\\
-{\gamma}\displaystyle\sum_{1\le i\ne j\le N}\zeta_i\zeta_j\log |x_j-x_i|
& {\mbox{ if }} s=1/2,
\end{matrix}\right.
\\ {\mbox{and }}&&\Sigma(t,r):=\gamma\int_0^r \sigma(t,y)\,dy.
\end{eqnarray*}
In particular, if the external stress is independent of the time,
then the potential~$V=V_0$
is authonomous and the map~$t\mapsto V_0\big(x_1(t),\ldots,x_N(t)\big)$
is nonincreasing in time.

We also remark that the behavior of~$V_0$ at infinity changes
dramatically when the fractional parameter~$s$ crosses the threshold~$1/2$
(this is in agreement with the strongly nonlocal interactions
expected when~$s<1/2$, see~\cite{dfv}). Nevertheless the convexity
of the functions
$(0,+\infty)\ni r\mapsto r^{1-2s}/(2s-1)$ (when~$s\ne1/2$)
and~$-\log r$ (when~$s=1/2$), which appear in the definition of~$V_0$, holds
for all~$s\in(0,1)$.

Finally, to state our result, we recall that the (upper and lower)
semi-continuous envelopes of a function $v$ are defined as
\beqs v^*(t,x):=\displaystyle\limsup_{(t',x')\rightarrow (t,x)} v(t',x')\eeqs
and 
\beqs v_*(t,x):=\displaystyle\liminf_{(t',x')\rightarrow (t,x)} v(t',x').\eeqs
Our main result is the following:
\begin{thm}\label{mainthm} Assume that  \eqref{W}, \eqref{sigmaassump} and \eqref{vep0} hold, and let
\begin{equation}\label{1.19bis}
v_0(t,x)=\displaystyle\sum_{i=1}^N H(\zeta_i(x-x_i(t)))-K,\end{equation}
where $H$ is the Heaviside function and $(x_i(t))_{i=1,...,N}$ is the solution to \eqref{dynamicalsysNintro}. 

Then, for every $\ep>0$ there exists a unique solution 
$v_\ep$ to  \eqref{vepeq}. Furthermore, as $\ep\rightarrow 0^+$, the solution $v_\ep$  exhibits the following asymptotic behavior: 
\beq\label{limsupv^ep}
\displaystyle\limsup_{(t',x')\rightarrow (t,x) \atop \ep\rightarrow 0^+} v_\ep(t',x')\leq (v_0)^*(t,x)\eeq
and 
\beq\label{liminfv^ep}
\displaystyle\liminf_{(t',x')\rightarrow (t,x) \atop \ep\rightarrow 0^+} v_\ep(t',x')\geq (v_0)_*(t,x),\eeq
for any $(t,x)\in [0,T_c)\times\R$.
\end{thm}

We remark that equation~\eqref{vepeq} is not changed by adding an integer
constant to the solution, so subtracting $K$ in formula~\eqref{1.19bis}
(as well as in~\eqref{vep0} for consistency) is clearly unessential.
We chose this normalization in order to have that
$$ \lim_{x\to-\infty} v_0(t,x)=0 \ {\mbox{ and }} \
\lim_{x\to+\infty} v_0(t,x)=N-K.$$
That is, the dislocation function~$v_0$ is normalized
to start with value~$0$ at~$-\infty$. In this way, its value at~$+\infty$
is equal to the number of  the dislocations that are positive oriented.\medskip

When~$K=0$ (i.e. when all the dislocation are oriented
in the same direction), the result in Theorem~\ref{mainthm}
has been proven in~\cite{gonzalezmonneau, dpv, dfv}, so the novelty
of Theorem~\ref{mainthm} consists in treating the general case
in which the dislocations occur in possibly different orientation.
\medskip

The long time behavior of our problem in this case
is very different from the case of positive oriented transitions.
Indeed, in such situation, system~\eqref{dynamicalsysNintro} is
driven by a repulsive potential, i.e. the dislocations have the
tendency to repell each other, and the solution of~\eqref{dynamicalsysNintro}
is defined for all the times, see~\cite{FIM09}.
\medskip

On the other hand, when the dislocations do not have all the same
orientations, the potential in~\eqref{dynamicalsysNintro}
has two types of behaviors: it acts as a {\em repulsive} potential
for particles with the {\em same} orientation, and as an {\em attractive}
potential for particles with {\em opposite} orientations.
\medskip

This dichotomy between the repulsive and attractive properties
of the potential may lead to collisions, i.e. solutions
of~\eqref{dynamicalsysNintro} may cease to exist in a finite time,
due to the vanishing of the denominator. As far as we know,
the present literature does not offer a complete study
of system~\eqref{dynamicalsysNintro} and a full description
of the collision analysis is not available. Therefore
we present some concrete cases in which we can detect these
collisions and estimate explicitly the collision time.
\medskip

The first case that we treat in the details is the one
of two initial transitions
with opposite orientation,
i.e.~$N=2$ and~$K=1$ in~\eqref{vep0}. In this case,
we can estimate the collision time~$T_c$ when the external stress has a sign
and when the initial configuration is small (in dependence of
the stress), according to the following result.

\begin{thm}\label{THM 2}
Let~$N=2$ and~$K=1$. Let~$\theta_0:=x_2^0-x_1^0$. Then:
\begin{itemize}
\item If $\sigma(t,x)\le0$ for any~$t\ge0$ and any~$x\in\R$, then
$$ T_c\le \displaystyle\frac{s\theta_0^{2s+1}}{(2s+1)\gamma}.$$
\item If
\begin{equation}\label{coo}
\theta_0<\left(\displaystyle\frac{1}{2s\|\sigma\|_\infty}\right)^\frac{1}{2s},\end{equation}
then
$$ T_c\le \frac{s\theta_0^{1+2s}}{\gamma\,(2s\theta_0^{2s}\|\sigma\|_\infty-1)}.$$
\item Conversely, if~\eqref{coo} is violated, there are examples in which~$T_c=+\infty$.
\end{itemize}
\end{thm}

The next case of interest is when we have three initial dislocations
that have alternate orientations. In this case, we can show
that the collision time is finite if no external stress is present and
we can give explicit bounds on it. Also, triple collisions occur
in symmetric situations.

\begin{thm}\label{THM 3}
Let
$$ C_s:=\displaystyle\frac{2^{2s+1}}{2^{2s}-1} >1.$$
Let~$N=3$, $\zeta_1=\zeta_3=+1$ and~$\zeta_2=-1$, and assume that~$\sigma\equiv0$.

Let~$\theta_1^0:=x_2^0-x_1^0$ and~$\theta_2^0:=x_3^0-x_2^0$.
Then
$$ T_c\in \left[ \tau_c,\ C_s \tau_c\right], \ {\mbox{ with }} \
\tau_c:=\displaystyle\frac{s\min\{\theta_1^0,\theta_2^0\}^{2s+1}}{(2s+1)\gamma}.$$
Moreover, the functions~$\theta_1(t):=x_2(t)-x_1(t)$
and~$\theta_2(t):=x_3(t)-x_2(t)$ are order preserving in time,
i.e.
$$ {\mbox{if $\theta_1^0<\theta_2^0$ then $\theta_1(t)<\theta_2(t)$
for every $t\in[0,T_c]$.}}$$
Furthermore,  if~$\theta_1^0=\theta_2^0$, then
a triple collision occurs, namely
$\theta_1(t)=\theta_2(t)>0$ for every~$t\in[0,T_c)$, and
$$ \theta_1(T_c)=\theta_2(T_c)=0 \ {\mbox{ with }} \
T_c=
\displaystyle\frac{C_s\;s\;(\theta_1^0)^{2s+1}}{(2s+1)\gamma}.$$
Viceversa, if a triple collision occurs
at time~$T_c$, then~$\theta_1^0=\theta_2^0$
and~$T_c=
\displaystyle\frac{C_s\;s\;(\theta_1^0)^{2s+1}}{(2s+1)\gamma}$.
\end{thm}

Next,  let us go back to the case of two initial dislocations with opposite orientation, i.e. $N=2$ and $K=1$. 
Suppose that a collision occurs at a time $0<T_c<+\infty$, so that if $(x_1(t),x_2(t))$ is the solution of \eqref{dynamicalsysNintro}, then
$x_1(T_c)=x_2(T_c)=x_c$. Then \eqref{limsupv^ep} and \eqref{liminfv^ep} imply that for any $x\neq x_c$, we have
$$\displaystyle\lim_{t\rightarrow T_c^-}\lim_{\ep\rightarrow 0^+}v_\ep(t,x)=0.$$
This can be rephrased saying that after the collision, the two dislocations cancel each other. Nevertheless, the limit of $v_\ep(t,x)$ 
keeps memory of them, in the sence that~$v_\ep$ at the point of collision~$x_c$
does not vanish at the limit. Indeed, we have
\begin{thm}\label{twodislxcprop}Assume $N=2$ and $K=1$. Let $v_\ep$ be the solution to  \eqref{vepeq}, then
\beq\label{limisuequalcolllision}
\displaystyle\limsup_{t\rightarrow T_c^-\atop\ep\rightarrow 0^+} v_\ep(t,x_c)\geq1.\eeq
\end{thm}

In the next two results, that are Theorems~\ref{AX-0} and~\ref{tyuxfghk},
we deal with the case of~$N$ transitions (with, in general,~$N>3$).
It seems that the picture in this case can be extremely rich,
so we will focus on two concrete cases: when one of the initial
distance between dislocations is much smaller than the others, and
when the orientations of the dislocations are alternate.\medskip

For this, we assume $\sigma\equiv 0$ and, for $i=1,...,N-1$, we
consider the distance between two consecutive dislocations:
\begin{equation}\label{AXX}\begin{split}
&\theta_i(t):=x_{i+1}-x_{i}\\ {\mbox{and }}\quad& 
\theta_i^0:=x^0_{i+1}-x^0_{i}>0.\end{split}\end{equation}
 Then, recalling~\eqref{dynamicalsysNintro}, we have that the $\theta_i$'s satisfy 
 \beq\label{dynamicalsysthetaN}\begin{cases} \dot{\theta}_i=\displaystyle\frac{\gamma}{2s}\left(\frac{2\zeta_i\zeta_{i+1}}{\theta_i^{2s}}+\displaystyle\sum_{j=1}^{i-1}\zeta_{i+1}\zeta_j \displaystyle\frac{1}{ (x_{i+1}-x_j)^{2s}}-\displaystyle\sum_{j=i+2}^{N}\zeta_{i+1}\zeta_j \displaystyle\frac{1}{ (x_j-x_{i+1})^{2s}}\right.\\
 \quad\left. -\displaystyle\sum_{j=1}^{i-1}\zeta_{i}\zeta_j \displaystyle\frac{1}{ (x_{i}-x_j)^{2s}}+\displaystyle\sum_{j=i+2}^{N}\zeta_{i}\zeta_j \displaystyle\frac{1}{ (x_j-x_{i})^{2s}} \right)&\text{in }(0,T_c)\\
 \theta_i(0)=\theta_i^0,
\end{cases}\eeq
$i=1,...,N-1$. Then we show that if two transitions with opposite orientations
are sufficiently close at the initial time, then a collision
in finite time occurs:

\begin{thm}\label{AX-0} Assume $N\ge 2$, $K\ge1$ and $\sigma\equiv 0$.
Then there exists $a_0\in(0,1)$
such that, if  for some $i=1,...,N-1$ 
\begin{eqnarray}
&& \label{AX-1}
\zeta_i\zeta_{i+1}=-1\\
{\mbox{and }}&&\label{AX-2}
\theta_i^0\leq a_0\min_{j\neq i}\theta_j^0,\end{eqnarray}
then 
\begin{equation}\label{AX-3}
\theta_i(t)\leq a_0\min_{j\neq i}\theta_j(t)\quad\text{for any }t>0.\end{equation}
Moreover~$\theta_i$ goes to zero in a finite time $T_c$, with 
\beq\label{Tcthetaismall}T_c\leq\frac{s(\theta_i^0)^{2s+1}}{(2s+1)\gamma[1-(N-2)a_0^{2s}]}.\eeq
\end{thm}

Some observations on Theorem~\ref{AX-0} are in order.
First of all, condition~\eqref{AX-1} states that the orientations
of the $i$th and~$(i+1)$th dislocations have opposite signs,
and~\eqref{AX-2} means that the initial distance between these dislocation
is small (when compared with the other dislocation distances).
Then, we obtain in~\eqref{AX-3} that this smallness and order condition
on the distances is preserved in time. 

Also, we remark that the estimate of the
collision time obtained in~\eqref{Tcthetaismall}
is somehow sharp, since it
reduces to the one in Theorem~\ref{THM 2} when~$N=2$.
\medskip

Next result deals with the alternating case, i.e. the case
in which after any dislocation we have a dislocation with the opposite orientation.
In this case, collisions occur, and we can estimate the collision time
according to the following result:

\begin{thm}\label{tyuxfghk} Assume $\sigma\equiv 0$ and \begin{equation}\label{alt}
\zeta_i\zeta_{i+1}=-1\end{equation} for any $i=1,...,N-1$. Then a collision occurs
in a finite time $T_c$, with
\beqs T_c\leq \frac{(N-1)(x^N_0-x^1_0)^{2s+1}}{(2s+1)\gamma}\quad\text{if }N\text{ is odd},\eeqs
and 
\beqs T_c\leq \frac{s(x^N_0-x^1_0)^{2s+1}}{(2s+1)\gamma}\quad\text{if }N\text{ is even}.\eeqs
\end{thm}

Notice that condition~\eqref{alt} says that the dislocations
have an alternate orientation (i.e. if the~$i$th dislocation
is positive oriented, then the~$(i+1)$th is negative oriented,
and viceversa).

We observe that the collision times obtained in Theorem~\ref{tyuxfghk}
is bounded by the initial maximal dislocation distance to the power~$2s+1$.
This estimate is, in a sense, optimal, when compared with the explicit estimates
in Theorems~\ref{THM 2} and~\ref{THM 3}.\medskip

The rest of the paper is organized as follows.
First, in Section~\ref{PP}
we give some general preliminary
results and some heuristics which link the partial differential
equation in~\eqref{vepeq} with the system of ordinary differential
equations in~\eqref{dynamicalsysNintro}.

Then, we
deal with the analysis of the collisions of the dynamical system
in~\eqref{dynamicalsysNintro}, which has somehow an independent interest: 
we study the case of two, three
and~$N$ dislocations in Sections~\ref{Tr2}, \ref{Tr3} and~\ref{TrN},
respectively. In this way, we also complete the proofs of Theorems~\ref{THM 2},
\ref{THM 3}, \ref{AX-0} and~\ref{tyuxfghk}.

Then, in Section~\ref{profmainthmsec} we prove Theorems~\ref{mainthm}
and~\ref{twodislxcprop}.

\section{Preliminary observations}\label{PP}

\subsection{Toolbox}
In this section we recall some general auxiliary results that  will be used in the rest of the paper. 
We recall that the existence of a unique solution of \eqref{u} is proven in \cite{cs}, while asymptotic estimates for $u$ and $u'$ are given  in  \cite{psv}. Finer estimates on $u$ are shown 
in \cite{dpv} and 
\cite{dfv} respectively when $s\in\left[\frac{1}{2},1\right)$ and $s\in\left(0,\frac{1}{2}\right)$. We collect these results in the following
 \begin{lem}\label{uinfinitylem} Assume that  \eqref{W} holds, then there exists a unique solution $u\in C^{2,\alpha}(\R)$. Moreover,
  there exists a  constant $C >0$ and $\kappa>2s$ (only depending on~$s$)
such that 
\begin{equation}\label{uinfinity}\left|u(x)-H(x)+\displaystyle\frac{1}{2sW''(0) 
}\displaystyle\frac{x}{|x|^{2s}}\right|\leq \displaystyle\frac{C}{|x|^{\kappa}},\quad\text{for }|x|\geq 1,
\end{equation} 
and
\begin{equation}\label{u'infinity}|u'(x)|\leq \displaystyle\frac{C}{|x|^{1+2s}}\quad\text{for }|x|\geq 1.
\end{equation} 
\end{lem}

Next, we introduce the function $\psi$ to be the solution of
\beq\label{psi}\begin{cases}\I\psi-W''(u)\psi=u'+\eta(W''(u)-W''(0))&\text{in }\R\\
\psi(-\infty)=0=\psi(+\infty),
\end{cases}\eeq where $u$ is the solution of \eqref{u} and
\beq\label{eta}\eta:=\displaystyle\frac{1}{W''(0)}
\displaystyle\int_\R(u'(x))^2dx
=\displaystyle\frac{1}{\gamma \beta}.\eeq
For a detailed heuristic motivation of such equation see Section 3.1 of \cite{gonzalezmonneau}.
The following results are  proven in \cite{dpv} and 
\cite{dfv}. 
\begin{lem} Assume that  \eqref{W} holds, then there exists a unique solution $\psi$ to \eqref{psi}. Furthermore 
$\psi\in C^{1,\alpha}_{loc}(\R)\cap L^\infty(\R)$ for some $\al\in(0,1)$ and $\psi'\in L^\infty(\R)$.
\end{lem}


\subsection{Heuristics of the dynamics}
We think that it could be useful to understand the
heuristic derivation of~\eqref{dynamicalsysNintro}
in the simpler setting of two particles with different orientations
(i.e.~$N=2$ and $K=1$).

For this, let $u$ be the solution of \eqref{u}. Let us introduce the notation
$$u_{\ep,1}(t,x):=u\left(\displaystyle\frac{x-x_1(t)}{\ep}\right),\quad u_{\ep,2}(t,x):=u\left(\displaystyle\frac{x_2(t)-x}{\ep}\right)-1,$$
and with a slight abuse of notation
$$u'_{\ep,1}(t,x):=u'\left(\displaystyle\frac{x-x_1(t)}{\ep}\right),\quad u'_{\ep,2}(t,x):=u'\left(\displaystyle\frac{x_2(t)-x}{\ep}\right).$$

Let us consider the following ansatz for $v_\ep$

$$v_\ep(t,x)\simeq \ueu(t,x)+\ued(t,x)=u\left(\displaystyle\frac{x-x_1(t)}{\ep}\right)+u\left(\displaystyle\frac{x_2(t)-x}{\ep}\right)-1.$$
Then, we compute

\beqs\begin{split}(v_\ep)_t&=-u'\left(\displaystyle\frac{x-x_1(t)}{\ep}\right)\displaystyle\frac{\dot{x}_1(t)}{\ep}+u'\left(\displaystyle\frac{x_2(t)-x}{\ep}\right)\displaystyle\frac{\dot{x}_2(t)}{\ep}\\
&=-u'_{\ep,1}(t,x)\displaystyle\frac{\dot{x}_1(t)}{\ep}+u'_{\ep,2}(t,x)\displaystyle\frac{\dot{x}_2(t)}{\ep},\end{split}\eeqs
and using the equation \eqref{u} and the periodicity of $W$
\beqs\begin{split}\I v_\ep(t,x)&=\displaystyle\frac{1}{\ep^{2s}}\I u \left(\displaystyle\frac{x-x_1(t)}{\ep}\right)+\displaystyle\frac{1}{\ep^{2s}}\I u \left(\displaystyle\frac{x_2(t)-x}{\ep}\right)\\&=
\displaystyle\frac{1}{\ep^{2s}} W'\left(u\left(\displaystyle\frac{x-x_1(t)}{\ep}\right)\right)+ \displaystyle\frac{1}{\ep^{2s}}W'\left(u\left(\displaystyle\frac{x_2(t)-x}{\ep}\right)\right)\\&=
\displaystyle\frac{1}{\ep^{2s}}W'(\ueu(t,x))+\displaystyle\frac{1}{\ep^{2s}}W'(\ued(t,x)).
\end{split}\eeqs

\noindent By inserting into \eqref{vepeq},  we obtain
\beq\label{dynamsisteqbeforeint}\begin{split}-u'_{\ep,1}\displaystyle\frac{\dot{x}_1}{\ep}+u'_{\ep,2}\displaystyle\frac{\dot{x}_2}{\ep}=\displaystyle\frac{1}{\ep^{2s+1}}\Big(W'(\ueu)+W'(\ued)-W'(\ueu+\ued)\Big)+\displaystyle\frac{\sigma}{\ep}.\end{split}\eeq

\noindent Now we make some observations on the asymptotics of the potential $W$. First of all, we notice that the periodicity of $W$ and the asymptotic behavior of $u$ imply
\beq\label{intwpup=0}\displaystyle\int_\R W'(u(x))u'(x)dx=\displaystyle\int_\R\displaystyle\frac{d}{dx} W(u(x))dx=W(1)-W(0)=0,\eeq and similarly
\beq\label{intwppup=0}\displaystyle\int_\R W''(u(x))u'(x)dx=0.\eeq

\noindent Next, we use estimate \eqref{uinfinitylem}  and make a Taylor expansion of $W'$ at 0 to compute for $x\neq x_2$
\beqs\begin{split}W'\left(u\left(\displaystyle\frac{x_2-x}{\ep}\right)\right)&\simeq W'\left(H\left(\displaystyle\frac{x_2-x}{\ep}\right)+\displaystyle\frac{\ep^{2s}(x-x_2)}{2s W''(0)|x-x_2|^{1+2s}}\right)\\&
=W'\left(\displaystyle\frac{\ep^{2s}(x-x_2)}{2s W''(0)|x-x_2|^{1+2s}}\right)\\&\simeq W''(0)\displaystyle\frac{\ep^{2s}(x-x_2)}{2s W''(0)|x-x_2|^{1+2s}}\\&
=\displaystyle\frac{\ep^{2s}(x-x_2)}{2s |x-x_2|^{1+2s}}.\end{split}\eeqs
So, we use the substitution $y=(x-x_1)/\ep$ to see that 
\beqs\begin{split}\displaystyle\frac{1}{\ep}\displaystyle\int_\R W'(\ued(t,x))\ueu'(t,x)dx&
\simeq \displaystyle\frac{1}{\ep}\displaystyle\int_\R \displaystyle\frac{\ep^{2s}(x-x_2)}{2s |x-x_2|^{1+2s}}u'\left(\displaystyle\frac{x-x_1}{\ep}\right)dx\\&
= \displaystyle\int_\R \displaystyle\frac{\ep^{2s}(\ep y+x_1-x_2)}{2s |\ep y+x_1-x_2|^{1+2s}}u'(y)dy\\&
\simeq  \displaystyle\frac{\ep^{2s}(x_1-x_2)}{2s |x_1-x_2|^{1+2s}} \displaystyle\int_\R u'(y)dy\\&
=\displaystyle\frac{\ep^{2s}(x_1-x_2)}{2s |x_1-x_2|^{1+2s}},
\end{split}\eeqs 
if $x_1\neq x_2$. Hence
\beq\label{wpu2u1p}\displaystyle\frac{1}{\ep^{2s+1}}\displaystyle\int_\R W'(\ued(t,x))\ueu'(t,x)dx\simeq\displaystyle\frac{x_1-x_2}{2s |x_1-x_2|^{1+2s}},\eeq
if $x_1\neq x_2$.
We use again  the substitution $y=(x-x_1)/\ep$, \eqref{intwpup=0} and \eqref{intwppup=0} to get
\beqs\begin{split} 
&\displaystyle\frac{1}{\ep} \displaystyle\int_\R W'(\ueu(t,x)+\ued(t,x))\ueu'(t,x)dx\\&
\simeq \displaystyle\frac{1}{\ep}\displaystyle\int_\R W'\left(u\left(\displaystyle\frac{x-x_1}{\ep}\right)+H(x)+\displaystyle\frac{\ep^{2s}(x-x_2)}{2s W''(0) |x-x_2|^{1+2s}}\right)u'\left(\displaystyle\frac{x-x_1}{\ep}\right)dx\\&
=\displaystyle\int_\R W'\left(u(y)+\displaystyle\frac{\ep^{2s}(\ep y+x_1-x_2)}{2s W''(0) |\ep y+x_1-x_2|^{1+2s}}\right)u'(y)dy\\&
\simeq\displaystyle\int_\R W'(u(y))u'(y)dy+\displaystyle\int_\R W''(u(y))\displaystyle\frac{\ep^{2s}(\ep y+x_1-x_2)}{2s W''(0) |\ep y+x_1-x_2|^{1+2s}}u'(y)dy\\&
\simeq \displaystyle\frac{\ep^{2s}(x_1-x_2)}{2s W''(0) |x_1-x_2|^{1+2s}}\displaystyle\int_\R W''(u(y))u'(y)dy\\&
=0.
\end{split}\eeqs

\noindent We deduce
\beq\label{wpu1u2u1p}
\displaystyle\frac{1}{\ep^{1+2s}}\displaystyle\int_\R W'(\ueu(t,x)+\ued(t,x))\ueu'(t,x)dx\simeq0.
\eeq

\noindent Moreover, we have 
\beq\label{sigmau1p}\begin{split}\displaystyle\frac{1}{\ep}\displaystyle\int_\R \sigma(t,x)\ueu'(t,x)dx&=\displaystyle\int_\R\sigma(t,\ep y+x_1)u'(y)dy\\&
\simeq \sigma(t,x_1)\displaystyle\int_\R u'(y)dy\\&
=\sigma(t,x_1).
\end{split}\eeq

\noindent Finally
\beq\label{u1p2}\displaystyle\frac{1}{\ep}\displaystyle\int_\R (\ueu'(t,x))^2dx=\displaystyle\int_\R (u'(y))^2dy=\gamma^{-1},\eeq
and using \eqref{u'infinity}
\beq\label{u1pu2p}\begin{split}\displaystyle\frac{1}{\ep}\displaystyle\int_\R \ueu'(t,x)\ued'(t,x)dx&\simeq\displaystyle\frac{1}{\ep}\displaystyle\int_\R u'\left(\displaystyle\frac{x-x_1}{\ep}\right)\displaystyle\frac{\ep^{1+2s}}{|x-x_2|^{1+2s}}dx\\&
=\displaystyle\int_\R u'(y)\displaystyle\frac{\ep^{1+2s}}{|\ep y+x_1-x_2|^{1+2s}}dy\\&
\simeq\displaystyle\frac{\ep^{1+2s}}{|x_1-x_2|^{1+2s}}\displaystyle\int_\R u'(y)dy\\& \simeq 0,\end{split}\eeq
if $x_1\neq x_2$. Now we multiply \eqref{dynamsisteqbeforeint} by $\ueu'(t,x)$, we integrate on $\R$ and we use  \eqref{intwpup=0}, \eqref{wpu2u1p}, \eqref{wpu1u2u1p}, \eqref{sigmau1p}, 
\eqref{u1p2} and \eqref{u1pu2p}, to get

\beqs-\gamma^{-1}\dot{x}_1=\displaystyle\frac{x_1-x_2}{2s |x_1-x_2|^{1+2s}}+\sigma(t,x_1).\eeqs
A similar equation is obtained if we multiply  \eqref{dynamsisteqbeforeint} by $\ued'(t,x)$ and integrate on $\R$.
Therefore we get the system

\beq\label{dynamicalsysheursec}\begin{cases} \dot{x}_1=-\gamma\displaystyle\frac{x_1-x_2}{2s |x_1-x_2|^{1+2s}}-\gamma\sigma(t,x_1)\\
\dot{x}_2=-\gamma\displaystyle\frac{x_2-x_1}{2s |x_2-x_1|^{1+2s}}+\gamma\sigma(t,x_2),
\end{cases}\eeq
which is \eqref{dynamicalsysNintro} with $N=2$ and $K=1$. 
This is a heuristic justification of the link between
the partial differential
equation in~\eqref{vepeq} and the system of ordinary differential
equations in~\eqref{dynamicalsysNintro}.

\section{Two transition layers: collision in finite time and proof of 
Theorem~\ref{THM 2}}\label{Tr2}

Let $(x_1(t),x_2(t))$ be the solution of \eqref{dynamicalsysheursec} with initial condition $x_1(0)=x_1^0<x_2(0)=x_2^0$. 
We want to show that under some assumptions on the external force $\sigma$ the time of collision between $x_1(t)$ and
$x_2(t))$ is finite and we also explicitly estimate its value. 
Let us denote
$$\theta(t):=x_2(t)-x_1(t),$$
$$\theta_0:=x_2^0-x_1^0>0,$$ then in an interval $(0,T_c)$, $\theta$ is solution of
\beq\label{thetaeq}\begin{cases}\dot{\theta}=-\displaystyle\frac{\gamma}{s\theta^{2s}}+\gamma\sigma(t,x_1)+\gamma\sigma(t,x_2)\\
\theta(0)=\theta_0>0.
\end{cases}\eeq
Let us first assume 
\beqs\sigma\leq 0.\eeqs
In this particular case, since $\theta$ is subsolution of 
\beq\label{sigmale0thetasol}\begin{cases}\dot{\theta}=-\displaystyle\frac{\gamma}{s\theta^{2s}}\\
\theta(0)=\theta_0,
\end{cases}\eeq
in the set where $\theta$ is positive, we have  $$\theta\leq \tilde{\theta},$$ where 
$$\tilde{\theta}(t):=\left(-\displaystyle\frac{2s+1}{s}\gamma t+\theta_0^{2s+1}\right)^\frac{1}{2s+1}$$
is the solution of \eqref{sigmale0thetasol}. 
The function  $\tilde{\theta}(t)$ vanishes for 
$t=\displaystyle\frac{s\theta_0^{2s+1}}{(2s+1)\gamma}$, therefore also $\theta$ vanishes in a finite time $T_c$ with
\beq\label{Tc2laybound}T_c\leq 
\displaystyle\frac{s\theta_0^{2s+1}}{(2s+1)\gamma}.\eeq
This gives the first claim in Theorem~\ref{THM 2}.\medskip

In the general case where no sign condition is assumed on $\sigma$, $\theta$ is subsolution of 
\beq\label{sigmagenthetasol}\dot{\theta}=-\displaystyle\frac{\gamma}{s\theta^{2s}}+2\gamma\|\sigma\|_\infty.\eeq

\noindent Equation \eqref{sigmagenthetasol} has the stationary solution 
$\theta_s(t):=\left(\displaystyle\frac{1}{2s\|\sigma\|_\infty}\right)^\frac{1}{2s}.$
 Therefore if~\eqref{coo} is satisfied,
since $\theta$ cannot touch $\theta_s$, its derivative remains negative. Hence 
$$\theta\leq\theta_0\quad\text{and}\quad \dot{\theta}<
-\displaystyle\frac{\gamma}{s\theta_0^{2s}}+2\gamma\|\sigma\|_\infty<0.$$ 
As a consequence, 
there exists a finite time $T_c$ such that $\theta(T_c)=0.$ 
More precisely, in this case
$$ \theta(t)\le \theta_0+t\,\left(
-\displaystyle\frac{\gamma}{s\theta_0^{2s}}+2\gamma\|\sigma\|_\infty\right)$$
and therefore
$$ T_c\le\frac{\theta_0}{2\gamma\|\sigma\|_\infty-(\gamma/s\theta_0^{2s})}
=\frac{s\theta_0^{1+2s}}{\gamma\,(2s\theta_0^{2s}\|\sigma\|_\infty-1)}.$$

This proves the second claim of Theorem \ref{THM 2}.

We also stress that if 
condition \eqref{coo}
is not satisfied (i.e. if~$\theta_0$ is not sufficiently small
with respect to the external stress), then~$\theta$ may never vanish
and~$T_c$ could be infinite. This is the case, for instance,   when $\sigma$ is a positive constant and~$\theta_0=1/(2s\sigma)^\frac{1}{2s}$.
This completes the proof of Theorem \ref{THM 2}.
\medskip


\section{Three transition layers: proof of Theorem~\ref{THM 3}}\label{Tr3}
Suppose that  we have three dislocations, two of them moving in the same direction while the central one moving in the opposite direction.
Then, system  \eqref{dynamicalsysNintro} becomes 
\beq\label{dynamicalsysbar3}\begin{cases} \dot{x}_1=\gamma\left(-\displaystyle\frac{x_1-x_2}{2s |x_1-x_2|^{1+2s}}+\displaystyle\frac{x_1-x_3}{2s |x_1-x_3|^{1+2s}}-\sigma(t,x_1)\right)\\
\dot{x}_2=\gamma\left(-\displaystyle\frac{x_2-x_1}{2s |x_2-x_1|^{1+2s}}-\displaystyle\frac{x_2-x_3}{2s |x_2-x_3|^{1+2s}}+\sigma(t,x_2)\right)\\
\dot{x}_3=\gamma\left(\displaystyle\frac{x_3-x_1}{2s |x_3-x_1|^{1+2s}}-\displaystyle\frac{x_3-x_2}{2s |x_3-x_2|^{1+2s}}-\sigma(t,x_3)\right)\\
x_1(0)=x_1^0<x_2(0)=x_2^0<x_3(0)=x_3^0.
\end{cases}\eeq

Let $(x_1(t),x_2(t),x_3(t))$ be the solution of \eqref{dynamicalsysbar3} and let us denote
\begin{equation*}\begin{split}
&\theta_{1}(t):=x_2(t)-x_1(t),\quad \theta_{2}(t):=x_3(t)-x_2(t),\\
&\theta_{1}^0:=x_2^0-x_1^0,\quad \theta_{2}^0:=x_3^0-x_2^0.\end{split}\end{equation*}
Then in the interval $(0,T_c)$, the function
$(\theta_{1},\theta_{2})$ is solution of
\beq\label{thetaeq3}\begin{cases}\dot{\theta}_{1}=\displaystyle\frac{\gamma}{s}\left(-\displaystyle\frac{1}{\theta_{1}^{2s}}+\displaystyle\frac{1}{2(\theta_{1}+\theta_{2})^{2s}}+\displaystyle\frac{1}{2\theta_{2}^{2s}}+\sigma(t,x_1)+\sigma(t,x_2)\right)\\
\dot{\theta}_{2}=\displaystyle\frac{\gamma}{s}\left(\displaystyle\frac{1}{2\theta_{1}^{2s}}+\displaystyle\frac{1}{2(\theta_{1}+\theta_{2})^{2s}}-\displaystyle\frac{1}{\theta_{2}^{2s}}-\sigma(t,x_3)-\sigma(t,x_2)\right)\\
\theta_{1}(0)=\theta_{1}^0>0\\
\theta_{2}(0)=\theta_{2}^0>0.\\
\end{cases}\eeq
Remark that in the particular case $\sigma\equiv 0$ and 
$$\theta_{1}^0=\theta_{2}^0=:\theta_0$$ the solution of system \eqref{thetaeq3} is given by
$$\theta_{1}(t)=\theta_{2}(t)=\theta(t)$$ where $\theta(t)$ is the solution of 

\beq\label{3laysystinitconeq}\begin{cases}\dot{\theta}=-\displaystyle\frac{\gamma}{2^{2s+1}s}\displaystyle\frac{2^{2s}-1}{\theta^{2s}}\\
\theta(0)=\theta_0>0.
\end{cases}\eeq
Integrating \eqref{3laysystinitconeq}, we get the following expression of $\theta$:
\beqs
\theta(t)=\left[\theta_0^{2s+1}-\gamma\displaystyle\frac{2s+1}{s}\displaystyle\frac{2^{2s}-1}{2^{2s+1}}t\right]^\frac{1}{2s+1}.\eeqs
We see that  $\theta$ vanishes at time \begin{equation}\label{T C}
T_c=\displaystyle\frac{2^{2s+1}}{2^{2s}-1}
\displaystyle\frac{s\theta_0^{2s+1}}{(2s+1)\gamma},\end{equation}
and we have a triple collision. 

Let us next show that if $\sigma\equiv 0$, 
for any choice of the initial condition $(x_1^0,x_2^0,x_3^0)$ we have a collision in a finite time,
and also that~$\theta_1$ and $\theta_2$ are order preserving, i.e.
if, for instance, 
\begin{equation}\label{th0}
\theta_1^0< \theta_2^0,\end{equation} then 
\beq\label{thet1lesstheta23lay}\theta_1(t)<\theta_2(t)\eeq
  for any  positive $t$ smaller than the collision time. Indeed, if there exists $t_0$ such that 
$\theta_1(t_0)=\theta_2(t_0)$, and we look at the solution $(\widetilde{\theta}_1(t),\widetilde{\theta}_2(t))$ of system \eqref{thetaeq3}  with initial condition 
$\theta_{1}^0=\theta_{2}^0=\theta_1(t_0)$, then  by the uniqueness of the solution of the system, we have
$$(\theta_1(t+t_0),\theta_2(t+t_0))=(\widetilde{\theta}_1(t),\widetilde{\theta}_2(t))$$ and we know that $\widetilde{\theta}_1(t)=\widetilde{\theta}_2(t)$ for any  $t$ smaller than the collision time. 
This is in contradiction with~\eqref{th0} and it proves~\eqref{thet1lesstheta23lay}.

In turn, inequality \eqref{thet1lesstheta23lay}  implies that $\theta_1(t)$ is subsolution of the equation \eqref{3laysystinitconeq} with initial condition 
$\theta_1(0)=\theta_{1}^0$. Therefore we have 
$$\theta_1(t)\leq \overline{\theta}_1(t):=\left[(\theta_{1}^0)^{2s+1}-\gamma\displaystyle\frac{2s+1}{s}\displaystyle\frac{2^{2s}-1}{2^{2s+1}}t\right]^\frac{1}{2s+1}.$$
In particular,  the collision time $T_c$ of the system \eqref{thetaeq3} is finite and  
$$T_c\leq \displaystyle\frac{2^{2s+1}}{2^{2s}-1}
\displaystyle\frac{s(\;\theta_1^0)^{2s+1}}{(2s+1)\gamma}.$$
Next, since $\theta_1(t)$ is supersolution of the equation \eqref{sigmale0thetasol}, we have
$$\theta_1(t)\geq \underline{\theta}_1(t):=\left[(\theta_1^0)^{2s+1}-\displaystyle\frac{2s+1}{s}\gamma t\right]^\frac{1}{2s+1}$$ and therefore
$$T_c\geq \displaystyle\frac{s\;(\theta_1^0)^{2s+1}}{(2s+1)\gamma}.$$
Finally, suppose that a triple collision occurs at some time~$T_c$.
We want to show that~$\theta_1(t)=\theta_2(t)$ for all~$t\in[0,T_c)$
and determine~$T_c$. For this, suppose, by contradiction, that
\begin{equation}\label{6tg}\theta_1(t_0)<
\theta_2(t_0).\end{equation} 
Then, by considering~$t_0$ the initial time of the flow,
we deduce from~\eqref{thet1lesstheta23lay} that~$\theta_1(t)<\theta_2(t)$
for every~$t\in [t_0,T_c)$. Using this and~\eqref{thetaeq3}, we see that
$$ \dot\theta_2-\dot\theta_1=\frac{\gamma}{s}\left( \frac{3}{2\theta_1^{2s}}
-\frac{3}{2\theta_2^{2s}}\right)>0$$
for every~$t\in [t_0,T_c)$. As a consequence, for any fixed~$a\in(0,T_c)$,
$$ (\theta_2-\theta_1)(T_c-a)
=(\theta_2-\theta_1)(t_0)+\int_{t_0}^{T_c-a} (\dot\theta_2-\dot\theta_1)(t)\,dt
>(\theta_2-\theta_1)(t_0).$$
This and~\eqref{6tg} are in contradiction with the fact that
$$ \lim_{a\to0^+}(\theta_2-\theta_1)(T_c-a)=0,$$
and so we have proved that~$\theta_1(t)=\theta_2(t)$ for all~$t\in[0,T_c)$.
In particular, we have that~$\theta_{1}^0=\theta_{2}^0$ and so the collision
time is determined by~\eqref{T C}.
This completes the proof of
Theorem~\ref{THM 3}.

\begin{rem} {\rm If the three dislocations are not alternated, i.e., $x_1$ and $x_2$ move in the same direction, while $x_3$ in the opposite one, then $x_2$ and $x_3$ collide in a finite time $T_c$ satisfying \eqref{Tc2laybound}. Indeed, in this case the repulsion between $x_1$ and $x_2$ and the attraction between $x_2$ and $x_3$ contribute positively to the collision.
}\end{rem}

\section{$N$ transition layers: some special cases and
proof of Theorems~\ref{AX-0}
and~\ref{tyuxfghk}}\label{TrN}

Now we deal with the case of $N$ transition layers.
Since the general picture can be very rich to describe,
we focus on the cases of small initial configuration and
alternate orientations, and we prove Theorems~\ref{AX-0}
and~\ref{tyuxfghk}.

\subsection{Proof of Theorem~\ref{AX-0}}
We fix $a_0>0$ small enough such that
\beq\label{C6}
-1+(N-2)a_0^{2s}+(N-1)a_0^{2s+1}<0.
\eeq
Let us denote 
$$\theta_{\rm{m}}(t):=\min_{j\neq i}\theta_j(t).$$ 
Of course, no confusion arises between the subscript~${\rm{m}}$,
that denotes this minimization and the indices~$i$ and~$j$.
Also, by~\eqref{AX-2} and~\eqref{C6}, we have that
\beq\label{asmall}
-1+(N-2)\frac{\theta_i^{2s}(0)}{\theta_{\rm{m}}^{2s}(0)}\leq -1+(N-2)a_0^{2s}<0.
\eeq
We want to show that for any $t>0$ 
\beq\label{thetai/thetamt>0}\frac{\theta_i(t)}{\theta_{\rm{m}}(t)} \leq a_0.\eeq
From system \eqref{dynamicalsysthetaN}, we infer that $\theta_i$ satisfies
\beq\label{thetaiderest}\dot{\theta_i}(t)\leq\frac{\gamma}{s}\left(-\frac{1}{\theta_i^{2s}}+\frac{N-2}{\theta_{\rm{m}}^{2s}}\right)
=\frac{\gamma}{s\theta_i^{2s}}\left(-1+(N-2)\frac{\theta_i^{2s}}{\theta_{\rm{m}}^{2s}}\right),\eeq while for any $j\neq i$
\beqs\dot{\theta_j}(t)\geq-\frac{\gamma(N-1)}{s\theta_{\rm{m}}^{2s}}.\eeqs 
{F}rom \eqref{asmall} and~\eqref{thetaiderest} we deduce that 
there exists $T>0$, that we choose maximal,  such that 
\begin{equation}\label{E1}
{\mbox{$\dot{\theta_i}(t)\leq 0$ for any $t\in(0,T)$. }}\end{equation}
Moreover, in $(0,T)$ we have that
\beqs\begin{split}\frac{d}{dt}{\left(\frac{\theta_i}{\theta_j}\right)}&=\frac{\dot{\theta_i}\theta_j-\theta_i\dot{\theta_j}}{\theta_j^2}\\&
\leq \frac{\dot{\theta_i}\theta_{\rm{m}}-\theta_i\dot{\theta_j}}{\theta_j^2}\\&
\leq\frac{\gamma}{s\theta_j^2}\left(-\frac{\theta_{\rm{m}}}{\theta_i^{2s}}+\frac{(N-2)\theta_{\rm{m}}}{\theta_{\rm{m}}^{2s}}+\frac{(N-1)\theta_i}{\theta_{\rm{m}}^{2s}}\right)\\&
=\frac{\gamma \theta_{\rm{m}}}{s\theta_j^2 \theta_i^{2s}}\left(-1+(N-2)\frac{\theta_i^{2s}}{\theta_{\rm{m}}^{2s}}+(N-1)\frac{\theta_i^{2s+1}}{\theta_{\rm{m}}^{2s+1}}\right).
\end{split}\eeqs
Integrating in $(0,t)$ and passing to the minimum  on $j$, we infer that for any $t\in(0,T)$
\beq\label{thetaifracthetam} \frac{\theta_i(t)}{\theta_{\rm{m}}(t)}\leq 
a_0+\min_j\int_0^t
\frac{\gamma \theta_{\rm{m}}(\tau)}{
s\theta^2_j(\tau) \theta_i^{2s}(\tau)}
\left(-1+(N-2)\frac{\theta_i^{2s}(\tau)}{\theta_{\rm{m}}^{2s}(\tau)}+(N-1)\frac{\theta_i^{2s+1}(\tau)}{\theta_{\rm{m}}^{2s+1}(\tau)}\right)d\tau.\eeq
Let us call
$$g(\tau):=-1+(N-2)\frac{\theta_i^{2s}(\tau)}{\theta_{\rm{m}}^{2s}(\tau)}+(N-1)\frac{\theta_i^{2s+1}(\tau)}{\theta_{\rm{m}}^{2s+1}(\tau)}.$$
We observe that~$g(0)<0$, thanks to~\eqref{C6}.
Thus, we want to show that 
\begin{equation}\label{AX-7}
{\mbox{$g(\tau)<0$ for any $\tau\in(0,T)$.}}\end{equation}
Assume by contradiction that this is not true.
Then there exists $t_0\in(0,T)$ such that 
\begin{equation}\label{AX-7bis}{\mbox{$g(\tau)<0$ for $\tau\in(0,t_0)$}}\end{equation} and 
$g(t_0)=0$. Then $\frac{\theta_i(t_0)}{\theta_{\rm{m}}(t_0)}=k$ with 
\begin{equation}\label{ACX-8}
-1+(N-2)k^{2s}+(N-1)k^{2s+1}=0.\end{equation} On the other hand, by~\eqref{thetaiderest} and~\eqref{AX-7bis},
we see that
$$\dot{\theta_i}<\frac{\gamma}{s\theta_i^{2s}}g< 0$$ 
in $(0,t_0)$, and therefore, recalling~\eqref{E1},
we conclude that~$t_0<T$. In particular, we can use~\eqref{thetaifracthetam}
with~$t:=t_0$.

Thus, from~\eqref{thetaifracthetam} and~\eqref{AX-7bis} we infer that 
  $$k=\frac{\theta_i(t_0)}{\theta_{\rm{m}}(t_0)}\le a_0+
\min_j\int_0^t
\frac{\gamma \theta_{\rm{m}}(\tau)}{
s\theta^2_j(\tau) \theta_i^{2s}(\tau)}\,
g(\tau)\,d\tau< a_0.$$ 
This and~\eqref{ACX-8} give that
$$ 0=-1+(N-2)k^{2s}+(N-1)k^{2s+1}<
-1+(N-2)a_0^{2s}+(N-1)a_0^{2s+1}$$
and this is in contradiction with~\eqref{C6}.
Therefore we have completed the proof of~\eqref{AX-7}.

In turn, we see that~\eqref{thetaifracthetam}
and~\eqref{AX-7} imply~\eqref{thetai/thetamt>0}, and thus~\eqref{AX-3}.

Finally \eqref{thetai/thetamt>0} and \eqref{thetaiderest} yield that 
 $$\dot{\theta_i}(t)\leq -\frac{\gamma[1-(N-2)a_0^{2s}]}{s\theta_i^{2s}},$$ and therefore  $\theta_i$ goes to zero in a time $T_c$ satisfying \eqref{Tcthetaismall}.
 Thus the proof of Theorem~\ref{AX-0} is complete.
 
\subsection{Proof of Theorem~\ref{tyuxfghk}}
Without loss of generality, we can assume $\zeta_1=1$. 
Let us first assume $N$ odd. Then $\zeta_N=1$,
and from~\eqref{dynamicalsysNintro} and~\eqref{AXX} we get
\beqs\begin{split} \dot{x}_N-\dot{x}_1&=\frac{\gamma}{2s}\sum_{j=1}^{N-1}\frac{\zeta_j}{(x_N-x_j)^{2s}}+\frac{\gamma}{2s}\sum_{j=2}^{N}\frac{\zeta_j}{(x_j-x_1)^{2s}}\\&
=\frac{\gamma}{2s}\left[\frac{1}{(\theta_1+\dots+\theta_{N-1})^{2s}}-\frac{1}{(\theta_2+\dots+\theta_{N-1})^{2s}}+\dots+\frac{1}{(\theta_{N-2}+\theta_{N-1})^{2s}}-\frac{1}{\theta_{N-1}^{2s}}\right.
\\&
\left.-\frac{1}{\theta_1^{2s}}+\frac{1}{(\theta_1+\theta_{2})^{2s}}-\dots-\frac{1}{(\theta_1+\dots+\theta_{N-2})^{2s}}+\frac{1}{(\theta_1+\dots+\theta_{N-1})^{2s}}\right].
\end{split} \eeqs
So, for every~$\ell\in\{1,\dots,N-2\}$, we introduce the notation
\begin{eqnarray*}
&&\alpha_\ell:=\frac{1}{(\theta_{\ell+1}+\dots+\theta_{N-1})^{2s}}
-\frac{1}{(\theta_{\ell}+\dots+\theta_{N-1})^{2s}} \\ {\mbox{ and }} &&
\beta_\ell:=\frac{1}{(\theta_{1}+\dots+\theta_\ell)^{2s}}
-\frac{1}{(\theta_1+\dots+\theta_{\ell+1})^{2s}}.\end{eqnarray*}
In this way, we have that~$\alpha_\ell$, $\beta_\ell\ge0$ and
\begin{equation}\label{AXAP}
\dot{x}_N-\dot{x}_1=-\frac{\gamma}{2s}\sum_{\ell=1}^{N-2}
(\alpha_\ell+\beta_\ell).\end{equation}
Moreover, for any~$a$, $b\ge0$ and any~$\xi\in[0,1]$ we have that
\begin{equation}\label{XU1}
(a+b)\cdot\frac{(\xi a+b)^{2s-1}}{b^{2s}}\ge (\xi
a+b)\cdot\frac{(\xi a+b)^{2s-1}}{b^{2s}} =\frac{(\xi a+b)^{2s}}{b^{2s}}\ge1.
\end{equation}
Thus, using a Taylor expansion we see that
there exists~$\xi_\ell\in[0,1]$ such that
\begin{equation}\label{XU2}\begin{split}
\alpha_\ell \, &=
\frac{(\theta_{\ell}+\dots+\theta_{N-1})^{2s}-(\theta_{\ell+1}+\dots+\theta_{N-1})^{2s}
}{(\theta_{\ell+1}+\dots+\theta_{N-1})^{2s}
(\theta_{\ell}+\dots+\theta_{N-1})^{2s}} \\
&=\frac{2s(\xi_\ell\theta_{\ell}+\theta_{\ell+1}+\dots+\theta_{N-1})^{2s-1}
\theta_\ell
}{(\theta_{\ell+1}+\dots+\theta_{N-1})^{2s}
(\theta_{\ell}+\dots+\theta_{N-1})^{2s}}
\\ &\ge \frac{2s\theta_\ell
}{(\theta_{\ell}+\dots+\theta_{N-1})^{1+2s}},
\end{split}
\end{equation}
where we have used~\eqref{XU1} here with~$\xi:=\xi_\ell$,
$a:=\theta_\ell$ and~$b:=\theta_{\ell+1}+\dots+\theta_N$.

Similarly, using~\eqref{XU1} with~$a:=\theta_{\ell+1}$ and~$
b:=\theta_1+\dots+\theta_\ell$, we see that
\begin{equation}\label{XU3}
\beta_\ell\ge\frac{2s\theta_{\ell+1}}{(\theta_1+\dots+\theta_{\ell+1})^{1+2s}}.
\end{equation}
{F}rom~\eqref{XU2} and~\eqref{XU3} we obtain that
\begin{equation}\label{XU4}
\alpha_\ell+\beta_\ell\ge\frac{2s\,
(\theta_\ell+\theta_{\ell+1})}{(\theta_1+\dots+\theta_{N-1})^{1+2s}}.
\end{equation}
Now, for any fix $t>0$ let $j(t)\in \{1,\dots,N-1\}$ be such that
$$ \theta_{j(t)}(t)=\max_{j=1,...,N-1}\theta_j(t).$$ 
Then, at time~$t$ we have that
$$ \theta_1+\dots+\theta_{N-1} \le (N-1) \theta_{j(t)}$$
and so~\eqref{XU4} implies that
$$ \alpha_\ell+\beta_\ell\ge\frac{2s\,
(\theta_\ell+\theta_{\ell+1})}{(N-1)\,\theta_{j(t)}\,
(\theta_1+\dots+\theta_{N-1})^{2s}},$$
for every~$\ell\in\{1,\dots,N-2\}$. In particular,
we can choose either~$\ell(t):=j(t)$ (if~$j(t)\ne N-1$)
or~$\ell(t):=j(t)-1$ (if~$j(t)=N-1$) and obtain that
\begin{eqnarray*}
\alpha_{\ell(t)}+\beta_{\ell(t)}&\ge&\frac{2s\,
\theta_{j(t)}}{(N-1)\,\theta_{j(t)}\,
(\theta_1+\dots+\theta_{N-1})^{2s}}\\
&=&\frac{2s}{(N-1)\,
(\theta_1+\dots+\theta_{N-1})^{2s}}.\end{eqnarray*}
This and~\eqref{AXAP} yield that, for any time~$t$ before collisions, we have
\beqs\dot{x}_N-\dot{x}_1\leq  -\frac{\gamma}{(N-1)(x_N-x_1)^{2s}}.\eeqs
 Since the solution of 
 \beqs\begin{cases}\dot{\theta}=-\displaystyle\frac{\gamma}{(N-1)\theta^{2s}}\\
\theta(0)=\theta_0>0
\end{cases}\eeqs
vanishes at the time $t=\frac{(N-1)\theta_0^{2s+1}}{(2s+1)\gamma}$, we can conclude that a collision ocurs at some time $T_c$ with 
$$T_c\leq\frac{(N-1)(x^N_0-x^1_0)^{2s+1}}{(2s+1)\gamma}.$$
The case $N$ even is simpler, thanks to direct cancellations. Indeed in this case,
from~\eqref{dynamicalsysNintro} and~\eqref{AXX}, we have
\beqs\begin{split} \dot{x}_N-\dot{x}_1&=-\frac{\gamma}{2s}\sum_{j=1}^{N-1}\frac{\zeta_j}{(x_N-x_j)^{2s}}+\frac{\gamma}{2s}\sum_{j=2}^{N}\frac{\zeta_j}{(x_j-x_1)^{2s}}\\&
=\frac{\gamma}{2s}\left[-\frac{1}{(\theta_1+\dots+\theta_{N-1})^{2s}}+\dots+\frac{1}{(\theta_{N-2}+\theta_{N-1})^{2s}}-\frac{1}{\theta_{N-1}^{2s}}\right.
\\&
\left.-\frac{1}{\theta_1^{2s}}+\frac{1}{(\theta_1+\theta_{2})^{2s}}-\dots-\frac{1}{(\theta_1+\dots+\theta_{N-1})^{2s}}\right]\\&
\leq -\frac{\gamma}{s(\theta_1+\dots+\theta_{N-1})^{2s}}\\&
=-\frac{\gamma}{s(x_N-x_1)^{2s}}.
\end{split} \eeqs
Therefore, a collision occurs in a time $T_c$ with 
$$T_c\leq \frac{s(x^N_0-x^1_0)^{2s+1}}{(2s+1)\gamma},$$
which completes the proof of Theorem~\ref{tyuxfghk}.


\section{Proof of Theorems \ref{mainthm}
and~\ref{twodislxcprop}}\label{profmainthmsec}

\subsection{Proof of Theorem~\ref{mainthm}}
As in~\cite{gonzalezmonneau, dpv, dfv}, the proof of Theorem \ref{mainthm}
relies on the construction of suitable barriers that allow
the use of Perron's method. Since in our case
the different transitions not need to be all oriented
in the same direction, some care is needed in order to take
into account the cancellations arising from the different
signs of the~$\zeta_i$'s.

More concretely, to prove the asymptotic behavior of $v_\ep$, namely
inequalities \eqref{limsupv^ep} and \eqref{liminfv^ep}, we construct suitable sub and supersolutions of \eqref{vepeq}. 
We consider an auxiliary small parameter $\delta>0$ and define $(\xs_1(t),...,\xs_{N}(t))$ to be the solution of the system 
\beq\label{dynamicalsysbarN}\begin{cases} \dot{\xs}_i=\gamma\left(\displaystyle\sum_{j\neq i}\zeta_i\zeta_j \displaystyle\frac{\xs_i-\xs_j}{2s |\xs_i-\xs_j|^{1+2s}}-\zeta_i\sigma(t,\xs_i)-\zeta_i\delta\right)&\text{in }(0,T_c-t_\delta)\\
 \xs_i(0)=x_i^0-\zeta_i\delta,
\end{cases}\eeq
$i=1,...,N$. Here $T_c$ is the collision time of the system \eqref{dynamicalsysNintro}. If we call $T^\delta_c$ the collision time of  the perturbed system 
\eqref{dynamicalsysbarN},  then 
\begin{equation}\label{fo}
\displaystyle\liminf_{\delta\rightarrow 0^+}T^\delta_c\ge T_c.\end{equation}
To check this, fix~$a\in(0,T_c)$, to be taken arbitrarily small
in the sequel. Then the solution of system~\eqref{dynamicalsysNintro} satisfies
$$ m_a:=\min_{{t\in [0,T_c-a]}\atop{1\le i\ne j\le N}}|x_i(t)-x_j(t)| >0.$$
Accordingly the right hand side of the equation in~\eqref{dynamicalsysNintro}
(together with its derivatives) is bounded when~$t\in [0,T_c-a]$
by a quantity that depends on~$a$.
Therefore, we are in the position to apply the continuity result
of the solution with respect to the parameter~$\delta$: we obtain
that there exists~$\delta_a>0$ such that, when~$\delta\in(0,\delta_a)$
the trajectories of~\eqref{dynamicalsysbarN} lie in a $(m_a/2)$-neighborhood
of the trajectories of~\eqref{dynamicalsysNintro}.
In particular, for any~$\delta\in(0,\delta_a)$, we have that
$$ \min_{{t\in [0,T_c-a]}\atop{1\le i\ne j\le N}}|\overline x_i(t)-\overline x_j(t)|
\ge \frac{m_a}{2}$$
and so the corresponding collision time cannot occur before~$T_c-a$.
That is~$T_c^\delta\ge T_c-a$ for all~$\delta\in(0,\delta_a)$, and so
$$ \liminf_{\delta\to0^+} T_c^\delta\ge T_c-a.$$
By taking~$a$ as close as we wish to~$0$, we obtain~\eqref{fo}.

In light of~\eqref{fo}, for $\delta$ small enough, we have that \eqref{dynamicalsysbarN} is well defined in $(0,T_c-t_\delta)$ 
where $t_\delta\rightarrow 0^+$ as $\delta\rightarrow 0^+$. 
Next,  we set 
\beq\label{xbarpuntoN}\cs_i(t):= \dot{\xs}_i(t),\quad i=1,...,N\eeq
and
\beq\label{barsigmaN}\overline{\sigma}:=\displaystyle\frac{\sigma+\delta}{W''(0)}.\eeq

\noindent Let $u$ and $\psi$ be  respectively the solution of \eqref{u} and \eqref{psi}. We define
\beq\label{vepansbarN}\begin{split}\vs_\ep(t,x)&:=\ep^{2s}\overline{\sigma}(t,x)-K+\displaystyle\sum_{i=1}^N u\left(\zeta_i\displaystyle\frac{x-\xs_i(t)}{\ep}\right)
-\displaystyle\sum_{i=1}^N\zeta_i\ep^{2s}\cs_i(t)\psi\left(\zeta_i\displaystyle\frac{x-\xs_i(t)}{\ep}\right).\end{split}
\eeq
 In order to simplify the notation, 
we set, for $i=1,...,N$
\beq\label{utildeN}\tilde{u}_i(t,x):=u\left(\zeta_i\displaystyle\frac{x-\xs_i(t)}{\ep}\right)-H\left(\zeta_i\displaystyle\frac{x-\xs_i(t)}{\ep}\right),\eeq
and 
\beqs \psi_i(t,x):=\psi\left(\zeta_i\displaystyle\frac{x-\xs_i(t)}{\ep}\right).\eeqs

\noindent  Finally, let
\beq\label{vepansbarbisN}I_\ep:=\ep(\vs_\ep)_t+\displaystyle\frac{1}{\ep^{2s}}(W'(\vs_\ep)-\ep^{2s}\I \vs_\ep-\ep^{2s}\sigma).\eeq

\noindent  The next two propositions show that $\vs_\ep$ is a supersolution of \eqref{vepeq}.
\begin{prop}\label{vbarepsupers} For  any $T<T_c-t_\delta$, there exists $\ep_0>0$ such that for   $0<\ep\leq\ep_0$, we have
$$(\vs_\ep)_t\geq\displaystyle\frac{1}{\ep}\left(\I \vs_\ep-\displaystyle\frac{1}{\ep^{2s}}W'(\vs_\ep)+\sigma(t,x)\right)\quad\text{in }(0,T)\times\R.$$
\end{prop}
\begin{prop}\label{leminitialcondinequ} There exists $\delta_0>0$ such that, for every $0<\delta\le\delta_0$, we have
\beqs  \vs_\ep(0,x)\geq v_\ep^0(x)\quad\text{for any }x\in\R.\eeqs
\end{prop}

We have the following asymptotic behavior for $\vs_\ep$:
\begin{lem}\label{ansatzlimitN} For any $(t,x)\in [0,T_c) \times\R$, we have that 
\beqs \displaystyle\lim_{\delta\rightarrow 0^+}\displaystyle\limsup_{(t',x')\rightarrow (t,x)\atop\ep\rightarrow 0^+}\vs_\ep(t',x')\leq (v_0)^*(t,x).
\eeqs
\end{lem}

The proof  of Proposition \ref{vbarepsupers}
is postponed to the next Section~\ref{post},
to avoid interruptions in the flow of the main arguments, while for the proofs of Lemma \ref{ansatzlimitN} and Proposition \ref{leminitialcondinequ}  we  refer respectively to the proofs of Lemma 8.1 and Proposition 8.2 in \cite{dpv}. 

Let us now conclude the proof of Theorem \ref{mainthm}.
First remark that, for $\ep$ sufficiently small, the initial condition $v^0_\ep$ given in \eqref{vep0} satisfies
$$-(N+1)\leq v^0_\ep\leq N+1.$$ Moreover the functions
$$\underline{u}_\ep(t,x):=-(N+1)-K_\ep t\quad\text{and}\quad \overline{u}_\ep(t,x):=N+1+K_\ep t$$ 
where $$K_\ep:=\displaystyle\frac{1}{\ep^{1+2s}}\|W'\|_{L^\infty(\R)}+\displaystyle\frac{1}{\ep}\|\sigma\|_{L^\infty(\R)},$$
are respectively sub and supersolution of \eqref{vepeq}. Hence, the existence of a unique, continuous  solution $v_\ep$ of \eqref{vepeq} is guaranteed  by the Perron's method and the comparison principle.

Next, from Propositions \ref{vbarepsupers} and \ref{leminitialcondinequ},  and the comparison principle, for any $T<T_c$ there exist $\delta_0$ and $\ep_0$ such that for 
$0<\delta\leq\delta_0$ and $0<\ep\leq\ep_0$, we have
\beq\label{vep<vepbar}v_\ep(t,x)\leq \vs_\ep(t,x)\quad\text{for any }(t,x)\in(0,T)\times\R.\eeq Passing to the limit as $\ep\rightarrow 0^+$, recalling Lemma \ref{ansatzlimitN}
and taking $\delta$ as small as we wish in the end, we get \eqref{limsupv^ep} for any $(t,x)\in [0,T_c)\times\R$. 

Similarly, to prove  \eqref{liminfv^ep}, for 
$\delta>0$ small, we  define $(\underline{x}_1(t),...,\underline{x}_{N}(t))$ to be the solution of the system 
\beq\label{dynamicalsysunderbarN}\begin{cases} \dot{\underline{x}}_i=\gamma\left(\displaystyle\sum_{j\neq i}\zeta_i\zeta_j \displaystyle\frac{\underline{x}_i-\underline{x}_j}{2s |\underline{x}_i-\underline{x}_j|^{1+2s}}-\zeta_i\sigma(t,\underline{x}_i)+\zeta_i\delta\right)&\text{in }(0,T_c-t_\delta)\\
 \underline{x}_i(0)=x_i^0+\zeta_i\delta,
\end{cases}\eeq
$i=1,...,N$, and 

\beqs\begin{split}\underline{v}_\ep(t,x)&:=\ep^{2s}\frac{\sigma(t,x)-\delta}{W''(0)}-K+\displaystyle\sum_{i=1}^N u\left(\zeta_i\displaystyle\frac{x-\underline{x}_i(t)}{\ep}\right)
-\displaystyle\sum_{i=1}^N\zeta_i\ep^{2s}\dot{\underline{x}}_i(t)\psi\left(\zeta_i\displaystyle\frac{x-\underline{x}_i(t)}{\ep}\right).\end{split}
\eeqs
Then, one can prove that $\underline{v}_\ep$ is a subsolution of \eqref{vepeq}  and therefore
\beq\label{veplessvundelinep} v_\ep(t,x)\geq \underline{v}_\ep(t,x)\quad\text{for any }(t,x)\in(0,T)\times\R,\eeq and any $T<T_c$, and any $\delta$ and $\ep$ small enough. Passing to the limit as $\ep\rightarrow 0^+$ and then letting
$\delta\rightarrow 0^+$, we get   \eqref{liminfv^ep}, thus
ending the proof of Theorem \ref{mainthm}.


\subsection{Proof of Theorem~\ref{twodislxcprop}}
Let us take a sequence $(T_k)_k$ such that $T_k\rightarrow T_c^-$ as $k\rightarrow+\infty$. Then, from \eqref{veplessvundelinep} with $N=2$ and $K=1$, there exist
$\delta^0_k$ and $\ep^0_k$ such that for any $\delta\in(0,\delta^0_k]$ and $\ep\in(0,\ep^0_k]$ we have
\begin{equation}\label{5.10bis}
\begin{split} v_\ep(T_k,x_c)&\geq O(\ep^{2s})+u\left(\displaystyle\frac{x_c-\underline{x}_1(T_k)}{\ep}\right)+u\left(\displaystyle\frac{\underline{x}_2(T_k)-x_c}{\ep}\right)-1\\&
-\ep^{2s}\dot{\underline{x}}_1(T_k)\psi\left(\displaystyle\frac{x_c-\underline{x}_1(T_k)}{\ep}\right)+\ep^{2s}\dot{\underline{x}}_2(T_k)\psi\left(\displaystyle\frac{\underline{x}_2(T_k)-x_c}{\ep}\right).
\end{split}\end{equation}
We remark that~$x_1(t)<x_2(t)$ for any~$t\in(0,T_c)$, and that both~$x_1(t)$ and~$x_2(t)$
approach~$x_c$ as~$t\to T_c^-$. Consequently, by~\eqref{dynamicalsysNintro},
we see that
$$ \dot{x}_1\ge\gamma\left(
\frac{1}{2s (x_2-x_1)|^{2s}}-\|\sigma_x\|_{L^\infty([0,+\infty)\times\R)}\right)
\to+\infty$$
as~$t\to T_c^-$. Similarly, we have that~$\dot{x}_2\to-\infty$
as~$\to T_c^-$.

We deduce that~$x_1$ is definitely incrasing in time, and~$x_2$
definitely decreasing. In particular, we have that~$x_1(t)<x_c<x_2(t)$
when~$t$ is close enough to~$T_c$ (and so for~$t=T_k$ and~$k$ large enough).

Therefore, we can take~$\delta=\delta_k>0$ sufficiently small that
\begin{equation*}
{\underline{x}}_1(T_k) < x_c < {\underline{x}}_2(T_k).\end{equation*}
As a consequence,
we can choose $\ep=\ep_k>0$ so small that $$\frac{x_c-\underline{x}_1(T_k)}{\ep_k},\,\frac{\underline{x}_2(T_k)-x_c}{\ep_k}\rightarrow +\infty\quad\text{as }k\rightarrow+\infty.$$
Then, by~\eqref{dynamicalsysunderbarN}, we have that
$$\ep_k^{2s}\dot{\underline{x}}_1(T_k)=\frac{\gamma}{2s\left(\frac{\underline{x}_2(T_k)-x_c}{\ep_k}+\frac{x_c-\underline{x}_1(T_k)}{\ep_k}\right)^{2s}}+O(\ep_k^{2s})\rightarrow 0\quad\text{as }
k\rightarrow+\infty.$$ Similarly 
$$\ep_k^{2s}\dot{\underline{x}}_2(T_k)\rightarrow 0\quad\text{as }k\rightarrow+\infty.$$
Thus, recalling~\eqref{5.10bis}, we infer that 
$$ \limsup_{k\rightarrow+\infty}v_{\ep_k}(T_k,x_c)\geq 1.$$ This implies that 
$$\displaystyle\limsup_{t\rightarrow T_c^-\atop\ep\rightarrow 0^+} v_\ep(t,x_c)\geq1,$$ which concludes the proof of Theorem~\ref{twodislxcprop}.

\subsection{Proof of Proposition \ref{vbarepsupers}}\label{post}
Let us start with the following
\begin{lem} For any $T<T_c-t_\delta$  in $(0,T)\times\R$ we have,  for $i=1,...,N$ 
\beq\label{ieplemN}\begin{split}I_\ep&=O(\tilde{u}_i)(\ep^{-2s}\displaystyle\sum_{j\neq i}\tilde{u}_j+\overline{\sigma}+\zeta_i\cs_i\eta)+\delta\\&
+\displaystyle\sum_{j\neq i}\left(O(\psi_j)+O(\tilde{u}_j)+O(\ep^{-2s}\tilde{u}_j^2)\right)+O(\ep^{2s}),
\end{split} 
\eeq
where $O(\ep^{2s})$ depends on $T$ and $\delta$.
\end{lem}
\begin{proof}

Fix $i=1,...,N$. We have
\beq\label{vepbartN}\begin{split}\ep(\vs_\ep)_t&=\ep^{2s+1}\overline{\sigma}_t-\displaystyle\sum_{j=1}^N \zeta_j\cs_ju'\left(\zeta_j\displaystyle\frac{x-\xs_j}{\ep}\right)\\&
+\displaystyle\sum_{j=1}^N\left(-\zeta_j\ep^{2s+1}\dot{\cs}_j\psi\left(\zeta_j\displaystyle\frac{x-\xs_j}{\ep}\right)+\zeta_j\ep^{2s}\cs_j^2\psi'\left(\zeta_j\displaystyle\frac{x-\xs_j}{\ep}\right)\right)\\&
=-\displaystyle\sum_{j=1}^N\zeta_j\cs_ju'\left(\zeta_j\displaystyle\frac{x-\xs_j}{\ep}\right)+O(\ep^{2s}).\end{split}
\eeq

Next, using the periodicity of $W$ and a Taylor expansion of $W'$ at $\tilde{u}_i$, we compute:
\begin{equation}\label{wpvbarN}\begin{split}\ep^{-2s} W'(\vs_\ep)&=\ep^{-2s}W'\left(\ep^{2s}\overline{\sigma}+\displaystyle\sum_{j\neq i}\tilde{u}_j+\tilde{u}_i-\displaystyle\sum_{j\neq i}\zeta_j\ep^{2s}\cs_j\psi_j-\zeta_i\ep^{2s}\cs_i\psi_i\right)\\&
=\ep^{-2s}W'(\tilde{u}_i)+\ep^{-2s}W''(\tilde{u}_i)(\ep^{2s}\overline{\sigma}+\displaystyle\sum_{j\neq i}\tilde{u}_j-\displaystyle\sum_{j\neq i}\zeta_j\ep^{2s}\cs_j\psi_j-\zeta_i\ep^{2s}\cs_i\psi_i)\\&
+\displaystyle\sum_{j\neq i}O(\ep^{-2s}\tilde{u}_j^2)+O(\ep^{2s}).
\end{split}\eeq

Finally, we evaluate
\begin{equation}\label{ivbarN}\begin{split}\I \vs_\ep&=\ep^{2s}\I \overline{\sigma}+\ep^{-2s}\displaystyle\sum_{j\neq i}\I u\left(\zeta_j\displaystyle\frac{x-\xs_j}{\ep}\right)
+\ep^{-2s}\I u\left(\zeta_i\displaystyle\frac{x-\xs_i}{\ep}\right)\\&
-\displaystyle\sum_{j\neq i}\zeta_j\cs_j\I\psi\left(\zeta_j\displaystyle\frac{x-\xs_j}{\ep}\right)-\zeta_i\cs_i\I\psi\left(\zeta_i\displaystyle\frac{x-\xs_i}{\ep}\right)\\&
=O(\ep^{2s})+\ep^{-2s}\displaystyle\sum_{j\neq i}W'(\tilde{u}_j)+\ep^{-2s}W'(\tilde{u}_i)\\&
-\displaystyle\sum_{j\neq i}\zeta_j\cs_j\left[W''(\tilde{u}_j)\psi_j+u'\left(\zeta_j\displaystyle\frac{x-\xs_j}{\ep}\right)+\eta(W''(\tilde{u}_j)-W''(0))\right]\\&
-\zeta_i\cs_i\left[W''(\tilde{u}_i)\psi_i+u'\left(\zeta_i\displaystyle\frac{x-\xs_i}{\ep}\right)+\eta(W''(\tilde{u}_i)-W''(0))\right].
\end{split}\eeq

\noindent Summing \eqref{vepbartN}, \eqref{wpvbarN} and \eqref{ivbarN}, and noticing that the terms involving $u'$, and the term $$\ep^{-2s}W'(\tilde{u}_i)-\zeta_i\cs_i W''(\tilde{u}_i)\psi_i$$
appearing in both \eqref{wpvbarN} and \eqref{ivbarN}, 
 cancel, we get
 \beqs\begin{split}I_\ep&=\ep(\vs_\ep)_t+\ep^{-2s}W'(\vs_\ep)-\I \vs_\ep-\sigma\\&
=-\ep^{-2s}\displaystyle\sum_{j\neq i}W'(\tilde{u}_j)+W''(\tilde{u}_i)\left(\overline{\sigma}+\ep^{-2s}\displaystyle\sum_{j\neq i}\tilde{u}_j\right)+\displaystyle\sum_{j\neq i}\zeta_j\cs_j(W''(\tilde{u}_j)-W''(\tilde{u}_i))\psi_j\\&
+\displaystyle\sum_{j\neq i}\zeta_j\cs_j\eta(W''(\tilde{u}_j)-W''(0))+\zeta_i\cs_i\eta(W''(\tilde{u}_i)-W''(0))-\sigma+\displaystyle\sum_{j\neq i}O(\ep^{-2s}\tilde{u}_j^2)+O(\ep^{2s}).
\end{split}\eeqs

\noindent Now, since $W'(0)=0$, we use a Taylor expansion of $W'$ around 0, to see that
$$\ep^{-2s}\displaystyle\sum_{j\neq i}W'(\tilde{u}_j)=\ep^{-2s}\displaystyle\sum_{j\neq i}W''(0)\tilde{u}_j+\displaystyle\sum_{j\neq i}O(\ep^{-2s}\tilde{u}_j^2),$$ so that
\beqs\begin{split}I_\ep&=-\ep^{-2s}\displaystyle\sum_{j\neq i}W''(0)\tilde{u}_j+W''(\tilde{u}_i)\left(\overline{\sigma}+\ep^{-2s}\displaystyle\sum_{j\neq i}\tilde{u}_j\right)
+\displaystyle\sum_{j\neq i}\zeta_j\cs_j(W''(\tilde{u}_j)-W''(\tilde{u}_i))\psi_j\\&
+\displaystyle\sum_{j\neq i}\zeta_j\cs_j\eta(W''(\tilde{u}_j)-W''(0))+\zeta_i\cs_i\eta(W''(\tilde{u}_i)-W''(0))-\sigma+\displaystyle\sum_{j\neq i}O(\ep^{-2s}\tilde{u}_j^2)+O(\ep^{2s}).
\end{split}\eeqs

\noindent Next, we add and subtract the term $W''(0)\overline{\sigma}$ to get 
\beqs\begin{split}I_\ep&=\ep^{-2s}\displaystyle\sum_{j\neq i}(W''(\tilde{u}_i)-W''(0))\tilde{u}_j+(W''(\tilde{u}_i)-W''(0))\overline{\sigma}+\displaystyle\sum_{j\neq i}\zeta_j\cs_j(W''(\tilde{u}_j)-W''(\tilde{u}_i))\psi_j\\&
+\displaystyle\sum_{j\neq i}\zeta_j\cs_j\eta(W''(\tilde{u}_j)-W''(0))+\zeta_i\cs_i\eta(W''(\tilde{u}_i)-W''(0))+W''(0)\overline{\sigma}-\sigma\\&+\displaystyle\sum_{j\neq i}O(\ep^{-2s}\tilde{u}_j^2)+O(\ep^{2s}).
\end{split}\eeqs

\noindent Now, clearly $$\zeta_j\cs_j\eta(W''(\tilde{u}_j)-W''(0))=O(\tilde{u}_j),\quad W''(\tilde{u}_i)-W''(0)=O(\tilde{u}_i)$$ and $$\zeta_j\cs_j(W''(\tilde{u}_j)-W''(\tilde{u}_i))\psi_j=O(\psi_j).$$
Therefore, we conclude that
\beqs\begin{split}I_\ep&=O(\tilde{u}_i)(\ep^{-2s}\displaystyle\sum_{j\neq i}\tilde{u}_j+\overline{\sigma}+\zeta_i\cs_i\eta)+W''(0)\overline{\sigma}-\sigma\\&
+\displaystyle\sum_{j\neq i}\left(O(\psi_j)+O(\tilde{u}_j)+O(\ep^{-2s}\tilde{u}_j^2)\right)+O(\ep^{2s}).
\end{split}\eeqs 
By \eqref{barsigmaN}, we finally obtain  \eqref{ieplemN}.
\end{proof}

Let us now conclude the proof of Proposition \ref{vbarepsupers}. 
Recalling \eqref{vepansbarbisN}, it suffices to show that for any $x\in\R$ and $t<T$
\beq\label{ieppositiveN} I_\ep\geq 0\eeq
for $\delta$ and $\ep$ small enough.

\noindent\emph{Case 1.}
Suppose that there exists an index $i=1,...,N$ such that $x$ is close to $\xs_i(t)$ more than $\ep^\gamma$:
\beq\label{x-xs2leqeppow}|x-\xs_i(t)|\leq \ep^\gamma\quad\text{with }0<\gamma<\displaystyle\frac{\kappa-2s}{\kappa},\eeq
where $\kappa$ is given in Lemma \ref{uinfinitylem}.

Since the $\xs_j$'s are separated for $t<T$, we have for $j\neq i$
$$|x-\xs_j(t)|\geq |\xs_i(t)-\xs_j(t)|-|x-\xs_i(t)|\geq|\xs_i(t)-\xs_j(t)|- \ep^\gamma\geq\overline{\theta}>0,$$   for $\ep$ sufficiently  small, where  $\overline{\theta}$ is independent of $\ep$. Hence, from \eqref{uinfinity} and  \eqref{utildeN}, we get for $j\neq i$
\beqs\begin{split}&\left|\displaystyle\frac{\tilde{u}_j(t,x)}{\ep^{2s}}+\displaystyle\frac{\zeta_j}{2s W''(0)}\displaystyle\frac{x-\xs_j(t)}{|x-\xs_j(t)|^{1+2s}}\right|
\\&= \displaystyle\frac{1}{\ep^{2s}}\left|u\left(\zeta_j\displaystyle\frac{x-\xs_j(t)}{\ep}\right)-H\left(\zeta_j\displaystyle\frac{x-\xs_j(t)}{\ep}\right)+\zeta_j\displaystyle\frac{\ep^{2s}}{2s W''(0)}\displaystyle\frac{x-\xs_j(t)}{|x-\xs_j(t)|^{1+2s}}\right|\\&
\leq C\displaystyle\frac{\ep^\kappa}{\ep^{2s}}\displaystyle\frac{1}{|x-\xs_j(t)|^\kappa}\\&\leq C\ep^{\kappa-2s}.\end{split}\eeqs
Next, a Taylor expansion of the function $\displaystyle\frac{x-\xs_j(t)}{|x-\xs_j(t)|^{1+2s}}$ around $\xs_i(t)$, gives
\beqs\begin{split}\left|\displaystyle\frac{x-\xs_j(t)}{|x-\xs_j(t)|^{1+2s}}-\displaystyle\frac{\xs_i(t)-\xs_j(t)}{|\xs_i(t)-\xs_j(t)|^{1+2s}}\right|&\leq \displaystyle\frac{2s}{|\xi-\xs_j(t)|^{1+2s}}|x-\xs_i(t)|\leq C\ep^\gamma,\end{split}\eeqs
where $\xi$ is a suitable point lying on the segment joining $x$ to $\xs_i(t)$. 

The last two inequalities imply for $j\neq i$
\beqs\left|\displaystyle\frac{\tilde{u}_j}{\ep^{2s}}+\displaystyle\frac{\zeta_j}{2s W''(0)}\displaystyle\frac{\xs_i(t)-\xs_j(t)}{|\xs_i(t)-\xs_j(t)|^{1+2s}}\right|\leq C(\ep^\gamma+\ep^{\kappa-2s}).\eeqs

\noindent Therefore, from  \eqref{ieplemN}, we get that
\beq\label{iepsemifinalN}\begin{split} I_\ep&=O(\tilde{u}_i)\left(\displaystyle\sum_{j\neq i}\displaystyle\frac{-\zeta_j}{2s W''(0)}\displaystyle\frac{\xs_i(t)-\xs_j(t)}{|\xs_i(t)-\xs_j(t)|^{1+2s}}+\overline{\sigma}+\zeta_i\cs_i\eta\right)+\delta \\&
+\displaystyle\sum_{j\neq i}\left(O(\psi_j)+O(\tilde{u}_j)+O(\ep^{-2s}\tilde{u}_j^2)\right)+O(\ep^{2s})+O(\ep^\gamma)+O(\ep^{\kappa-2s}).
\end{split}\eeq

\noindent Now, we compute the term between parenthesis. From the definitions of $\cs_i$, $\eta$ and $\overline{\sigma}$ given respectively in \eqref{xbarpuntoN}, \eqref{eta} and 
\eqref{barsigmaN}, and the system of ODE's \eqref{dynamicalsysbarN},  we obtain
\beq\label{parenttermestiN}\begin{split}\displaystyle\sum_{j\neq i}\displaystyle\frac{-\zeta_j}{2s W''(0)}\displaystyle\frac{\xs_i(t)-\xs_j(t)}{|\xs_i(t)-\xs_j(t)|^{1+2s}}+\overline{\sigma}+\zeta_i\cs_i\eta&=\displaystyle\frac{\sigma(t,x)-\sigma(t,\xs_i(t))}{W''(0)}\\&
=O(|x-\xs_i(t)|)\\&=O(\ep^\gamma).\end{split}\eeq

\noindent Finally, from the estimates \eqref{uinfinity}  and the fact that $\displaystyle\lim_{|x|\rightarrow\pm\infty}\psi(x)=0$, we have for $j\neq i$
\beq\label{psiuuqestN}\tilde{u}_j,\,\ep^{-2s}\tilde{u}_j^2=O(\ep^{2s}),\quad\text{and}\quad \psi_j=O(1),\eeq as $\ep\rightarrow 0$.
From \eqref{iepsemifinalN}, \eqref{parenttermestiN} and \eqref{psiuuqestN}, we get that for $\ep$ small enough
$$I_\ep\geq\displaystyle\frac{\delta}{2},$$ which implies \eqref{ieppositiveN}.

\noindent\emph{Case 2.} Suppose that for  $i=1,...,N$ we have 
$$|x-\xs_i(t)|\geq \ep^\gamma.$$ In this case,  the estimate in \eqref{uinfinity} on $u$ implies for  $j=1,...,N$
\beqs\left|\displaystyle\frac{\tilde{u}_j}{\ep^{2s}}+\displaystyle\frac{\zeta_j}{2s W''(0)}\displaystyle\frac{x-\xs_j(t)}{|x-\xs_j(t)|^{1+2s}}\right|\leq C\displaystyle\frac{\ep^\kappa}{\ep^{2s}}\displaystyle\frac{1}{|x-\xs_j(t)|^\kappa}
\leq C\ep^{\kappa-2s-\gamma\kappa}.\eeqs
Moreover
$$\displaystyle\frac{1}{|x-\xs_j(t)|^{2s}}\leq \ep^{-2\gamma s}.$$

\noindent As a consequence, recalling \eqref{eta}, \eqref{barsigmaN} and \eqref{dynamicalsysbarN}
\beqs\begin{split} \ep^{-2s}\displaystyle\sum_{j\neq i}\tilde{u}_j+\overline{\sigma}+\zeta_i\cs_i\eta&=\displaystyle\sum_{j\neq i}\displaystyle\frac{\zeta_j}{2s W''(0)}\displaystyle\frac{x-\xs_j(t)}{|x-\xs_j(t)|^{1+2s}}+O(1)\\&
=O(\ep^{-2\gamma s}).
\end{split}\eeqs

\noindent Therefore, from  \eqref{ieplemN}, we have
\beqs I_\ep=\delta+O(\tilde{u}_i)O(\ep^{-2\gamma s})
+\displaystyle\sum_{j\neq i}\left(O(\psi_j)+O(\tilde{u}_j)+O(\ep^{-2s}\tilde{u}_j^2)\right)+O(\ep^{2s}).
\eeqs
Now, we observe that again from \eqref{uinfinity}, for $i=1,...,N$
$$\tilde{u}_i=O\left(\displaystyle\frac{\ep^{2s}}{|x-\xs_i|^{2s}}\right)=O\left(\displaystyle\frac{\ep^{2s}}{\ep^{2\gamma s}}\right)=O\left(\ep^{2s(1-\gamma)}\right).$$
 As a consequence
$$O(\ep^{-2s}\tilde{u}_j^2)=O\left(\ep^{2s(1-2\gamma)}\right),\quad\text{and}\quad O(\tilde{u}_i)O(\ep^{-2\gamma s})=O\left(\ep^{2s(1-2\gamma)}\right).$$
Again the asymptotic behavior of $\psi$ implies 
$$\psi_i=O(1).$$
We conclude that 
$$I_\ep=\delta+O(1).$$
Hence for  $\ep$ small enough, we have 
$$I_\ep\geq\displaystyle\frac{\delta}{2},$$ which again implies \eqref{ieppositiveN}.

\vfill


\bigskip 

\begin{thebibliography}{10}


\bibitem{cs}{\sc X. Cabr\'{e} and Y. Sire}, Nonlinear equations for fractional Laplacians II: existence, uniqueness,
and qualitative properties of solutions, {\em Trans. Amer. Math. Soc.}, to appear.


\bibitem{csm}{\sc X. Cabr\'{e} and J. Sol\`{a}-Morales}, Layer
solutions in a half-space for boundary reactions, {\em Comm. Pure
Appl. Math.}, {\bf 58} (2005) no. 12, 1678-1732.


\bibitem{dfv}{\sc S. Dipierro, A. Figalli and E. Valdinoci},
Strongly nonlocal dislocation dynamics in crystals,  
{\em Comm.  Partial Differential Equations}, to appear.


\bibitem{dpv}{\sc S. Dipierro, G. Palatucci and E. Valdinoci}, Dislocation 
dynamics in crystals: a macroscopic
theory in a fractional Laplace setting, 
{\em Comm. Math. Phys.}, to appear.


\bibitem{dnpv}{\sc E. Di Nezza, G. Palatucci and E. Valdinoci}, Hitchhiker's guide to fractional Sobolev spaces,
{\em Bull. Sci. math.}, {\bf 136} (2012), no. 5, 521-573. 


\bibitem{FIM09} {\sc N. Forcadel, C. Imbert, R. Monneau},
Homogenization 
of some particle systems with two-body interactions and of the
dislocation dynamics, {\em Discrete Contin. Dyn. Syst.}, {\bf 23} (2009),
no.~3, 785-826.


\bibitem{gonzalezmonneau}{\sc M. Gonz\'{a}lez and R. Monneau},
Slow motion of particle systems as a limit of a reaction-diffusion
equation with half-Laplacian in dimension one, {\em Discrete Contin. Dyn. Syst.}, {\bf  32} (2012),
no. 4, 1255-1286.


\bibitem{Nab97} {\sc F. R. N. Nabarro},
Fifty-year study of the Peierls--Nabarro stress, {\em
Mat. Sci.  Eng. A} {\bf 234--236}~(1997), 67-76.


\bibitem{psv}{\sc G. Palatucci, O. Savin and  E. Valdinoci}, Local and global minimizers for a variational energy
involving a fractional norm.
{\em  Ann. Mat. Pura Appl.}, (4) {\bf 192} (2013), no. 4, 673-718.

\bibitem{s}{\sc L. Silvestre}, {\em Regularity of the obstacle problem for a fractional power of the Laplace operator}, PhD thesis, University of Texas at Austin (2005).

\end{thebibliography}
\end{document}